\documentclass[letterpaper,11pt,anonymous]{article}

\usepackage[margin=1in]{geometry}
\usepackage{graphicx}

\usepackage{authblk}
\usepackage[linesnumberedhidden]{algorithm2e}
\usepackage[nottoc,numbib]{tocbibind}

\usepackage{enumitem}

\usepackage{tikz}
\usetikzlibrary{arrows,calc,shapes}
\usepackage{tikz-cd}

\tikzstyle{arc}=[->,shorten <=3pt, shorten >=3pt,
                 >=stealth, line width=1.1pt]
\tikzstyle{edge}=[shorten <=2pt, shorten >=2pt,
                  >=stealth, line width=1.1pt]
\tikzstyle{blueE}=[shorten <=2pt, shorten >=2pt,
                  >=stealth, line width=1.5pt, blue]
\tikzstyle{redE}=[shorten <=2pt, shorten >=2pt,
                  >=stealth, line width=1.5pt, dashed, red]
\tikzstyle{vertex}=[circle, fill=white, draw,
                    minimum size=7pt,
                    inner sep=0pt, outer sep=0pt]
\tikzstyle{redA}=[->, dashed, shorten <=3pt, shorten >=3pt, red, >=stealth,
                  line width=1.5pt]
\tikzstyle{blueA}=[->, shorten <=3pt, shorten >=3pt, blue, >=stealth,
                   line width=1.5pt]

\usepackage{float}

\usepackage{xcolor}

\usepackage[hidelinks]{hyperref} 

\usepackage{amsmath}
\usepackage{amssymb}
\usepackage{bbm} 
\usepackage{xfrac}

\usepackage{amsthm}

\usepackage{multicol}

\newtheorem{theorem}{Theorem}[section]
\newtheorem{lemma}[theorem]{Lemma}
\newtheorem{corollary}[theorem]{Corollary}

\newtheorem{observation}[theorem]{Observation}

\newtheorem{conjecture}[theorem]{Conjecture}

\theoremstyle{remark}

\theoremstyle{definition}
\newtheorem{example}{Example}

\newcommand{\calC}{\mathcal C}
\newcommand{\calB}{\mathcal B}

\newcommand{\calF}{\mathcal F}
\newcommand{\calA}{\mathcal A}
\newcommand{\calP}{\mathcal P}

\newcommand{\calT}{\mathcal T}

\newcommand{\bA}{\mathbb A}
\newcommand{\bB}{\mathbb B}
\newcommand{\bC}{\mathbb C}
\newcommand{\bD}{\mathbb D}
\newcommand{\bG}{\mathbb G}
\newcommand{\bH}{\mathbb H}
\newcommand{\bS}{\mathbb S}
\newcommand{\bF}{\mathbb F}
\newcommand{\bT}{\mathbb T}
\newcommand{\bX}{\mathbb X}

\newcommand{\ignore}[1]{}

\DeclareMathOperator{\CSP}{CSP}
\DeclareMathOperator{\CSPs}{CSPs}

\DeclareMathOperator{\Age}{Age}

\DeclareMathOperator{\SNP}{SNP}
\DeclareMathOperator{\Forb}{Forb}
\DeclareMathOperator{\NP}{NP}
\DeclareMathOperator{\NE}{NE}
\renewcommand{\P}{\textnormal{P}}
\DeclareMathOperator{\NEE}{NEE}
\DeclareMathOperator{\EE}{EE}

\DeclareMathOperator{\coNP}{coNP}

\DeclareMathOperator{\E}{E}
\DeclareMathOperator{\PO}{P}

\DeclareMathOperator{\full}{full}
\DeclareMathOperator{\diam}{diam}
\DeclareMathOperator{\pd}{pd}
\DeclareMathOperator{\NT}{NT} 

\DeclareMathOperator{\HER}{Her}
\DeclareMathOperator{\HerFO}{HerFO}

\DeclareMathOperator{\ESO}{ESO}

\newcommand{\PDdiam}[1]{\diam^{\pd}_{#1}}

\title{On the Computational Power of Extensional ESO\thanks{This project has received funding from the European Union
(Project POCOCOP, ERC Synergy Grant 101071674).
Views and opinions expressed are however those of the author(s) only and do not
necessarily reflect those of the European Union or the European Research
Council Executive Agency. Neither the European Union nor the granting
authority can be held responsible for them.}}

 \author[1]{Manuel Bodirsky\thanks{manuel.bodirsky@tu-dresden.de}}
 \author[1]{Santiago Guzm\'an-Pro\thanks{santiago.guzman\_pro@tu-dresden.de}}
 \affil[1]{Institut f\"ur Algebra, TU Dresden, Germany}

\begin{document}
\date{}

\maketitle
\begin{abstract}
\emph{Extensional ESO} is a fragment of existential second-order logic
(ESO) that captures the following family of problems. Given a fixed ESO
sentence $\Psi$ and an input structure $\mathbb A$ the task if to decide
whether there is an \emph{extension} $\bB$ of $\bA$ that  satisfies the
first-order part of $\Psi$, i.e.,  a structure ${\mathbb B}$ such  that
$R^{\mathbb A}\subseteq R^{\mathbb B}$ for every existentially quantified
predicate $R$ of $\Psi$,  and $R^{\mathbb A} = R^{\mathbb B}$ for every
non-quantified predicate $R$ of $\Psi$. In particular, extensional ESO
describes all pre-coloured finite-domain constraint satisfaction problems (CSPs).

In this paper we study the computational power of extensional ESO; 
we ask, \emph{for which problems in $\NP$ is there a polynomial-time equivalent problem in
extensional ESO?}.
One of our main results states that extensional ESO has the same computational power as
\emph{hereditary first-order logic}. We also characterize the computational
power of the fragment of extensional ESO with monotone universal first-order part in
terms of finitely bounded CSPs. These results suggest a rich computational power of
this logic, and we conjecture that extensional ESO captures NP-intermediate problems. We further
support this conjecture by showing that extensional ESO can express
current candidate  NP-intermediate problems such as Graph Isomorphism, and Monotone Dualization
(up to polynomial-time equivalence). On the other hand, another main result
proves that extensional ESO does not have the full computational power of NP: there are problems
in  NP that are not polynomial-time equivalent to a problem in extensional ESP (unless E=NE).

\end{abstract}

\tableofcontents


\setcounter{page}{1}

\section{Introduction} 

 \emph{First-order logic (FO)} has good computational properties: 
given a fixed first-order sentence $\phi$ it can be verified in deterministic
logarithmic space whether an input structure $\bA$ satisfies $\phi$. 
However, it has been realized long ago that FO has a very limited expressive power when
it comes to describing classes of finite structures.  For instance, there is no first-order
sentence which holds on a finite graph if and only if the graph does not contain cycles
(i.e., is a forest). In contrast, second-order logic is very powerful and captures the
entire  polynomial-time hierarchy~\cite{Immerman}.  Even restricted fragments
of second-order logic can be very expressive. For instance, Fagin~\cite{Fagin} 
showed that \emph{existential second-order logic (ESO)}  captures the complexity class NP. Taking
complements, \emph{universal second-order logic (USO)} captures coNP. 

Further restrictions of ESO  might no longer have the full expressive power, but still the full
\emph{computational} power of NP, in the sense that every problem in NP is polynomial-time equivalent to a
problem that can be expressed in the restricted logic; we will call such classes \emph{NP-rich}.%
\footnote{If we consider disjunctive, conjunctive, or Turing reductions, we call such classes \emph{conjunctive, disjunctive,} or \emph{Turing NP-rich}, respectively. These reductions are defined in Section~\ref{sect:prelims}.}
A prominent example of an NP-rich ESO fragment is \textit{SNP} (for \emph{strict NP}), which is defined as the
restriction of ESO to sentences with a universal first-order part. SNP is quite expressive, and contains for instance
all of Datalog~\cite{FederVardi}, but clearly it cannot express every class of finite structures in NP,  because
SNP can only describe
\emph{hereditary classes}, i.e., classes $\calC$ closed under taking induced substructures. However, Feder and
Vardi~\cite{FederVardi}  proved that SNP is NP-rich (see~\cite{BarsukovM23} for some of the missing details of
the proof given in~\cite{FederVardi}). In particular, it follows from Ladner's theorem~\cite{ladnerACM22} that
whenever a fragment $\mathcal L$ of ESO is NP-rich, then  $\mathcal L$ does not exhibit a P
versus NP-complete dichotomy. 

On the dichotomy side, Feder and Vardi showed that a fragment of ESO called
\emph{monotone monadic SNP (MMSNP)} has the same computational power as finite-domain 
CSPs. Hence, it follows from the finite-domain CSP dichotomy~\cite{BulatovFVConjecture,ZhukFVConjecture}
that MMSNP exhibits a P vs.\ NP-complete dichotomy.  Bienvenu,  ten Cate, Lutz, and 
Wolter~\cite{bienvenu2014} asked whether the dichotomy for MMSNP extends to a broader
fragment called \emph{guarded monotone SNP (GMSNP)}. In this direction,
Bodirsky, Kn\"auer, and Starke~\cite{bodirsky_asnp} showed that GMSNP should indeed
exhibit a dichotomy according to the Bodirsky--Pinsker conjecture~\cite[Conjecture 3.7.1]{Book}
--- yet, the GMSNP (conjectured) dichotomy seems elusive. 
It is worth mentioning that so far, the 
main technique to show that a fragment of ESO does not exhibit a P versus NP-complete dichotomy is
by showing that it is NP-rich~\cite{KunNesetril24}.

\paragraph{The story.} In this paper we initiate the study of the computational power
    of fragments of ESO beyond complexity dichotomies and NP-richness. We propose to
    find a \emph{natural} and \emph{decidable} syntactic fragments of ESO that is neither
    NP-rich nor exhibits a P vs.\ NP-complete dichotomy. To this end,
    we introduce \emph{extensional ESO}, and we exhibit several families of problems
    that can be expressed in extensional ESO, ranging from the world of CSPs to
    graph theory
    and to language theory. We then address the question
    ``what is the computational  power of extensional ESO?''.
    We show that extensional ESO has the same computational power as the recently introduced
    formalism \emph{hereditary first-order logic}. We also show that extensional ESO
    is not NP-rich, unless $\E = \NE$.   
    However, we conjecture that extensional ESO captures NP-intermediate problems.
    We support this conjecture by showing that some examples of candidate NP-intermediate
    problems such as Graph Isomorphism and
    Monotone Dualization are expressible in extensional ESO.
    We also show that a very large class of CSPs, namely the class of CSPs for
    \emph{finitely bounded structures}, can be expressed in extensional ESO;
    we believe that this class contains NP-intermediate problems. 
    
    In Figure~\ref{fig:landscape} we present an illustration of this story,
    and we now present a brief description of \emph{emphasized characters}
    (all missing details will be carefully introduced later in the paper).

\paragraph{The main character.} The central role of this paper is played by \emph{extensional ESO}.
An ESO $\tau$-sentence $\Psi$ is called \emph{extensional} if for every existentially quantified
predicate $R$ there is a distinct symbol $R'\in \tau$ of the same arity such that $\Psi$
contains a conjunct $\forall \overline{x} \big (R'(\bar x) \Rightarrow R(\bar x) \big )$, and $R'$ does
not appear anywhere else in $\Psi$. In this case we say that $R$ \emph{extends} $R'$.
Clearly, every problem described by an ESO sentence $\Phi$ is in $\NP$.
To present two simple examples, we consider \emph{digraph acyclicity}, and
\emph{precoloured $3$-coloring problem} from graph theory; further examples can be found in Section~\ref{sec:extensional-ESO}.

\begin{example}[Acyclic digraphs]\label{ex:DAG}
The class of strict linear orders can be easily axiomatized in universal first-order
logic with a binary relation symbol $<$. Let $\phi$ be such a sentence, and notice
that a digraph $\mathbb D$ is a acyclic if and only if is a spanning graph of a transitive tournament,
in other words, of a strict linear order. Hence, the class of acyclic digraphs is described by
the extensional ESO sentence
\[
    \exists {<}\; \forall x,y  \big( (E(x,y)\Rightarrow ~{<}(x,y))~\land~\phi \big).
\]
\end{example}

 \begin{example}[Pre-coloured 3-colourability]
 Given an input graph $G$ with a partial colouring $c\colon U\subseteq V\to \{R,G,B\}$, 
    the task is to decide whether there is a $3$-colouring of $G$ such that its restriction to $U$
    coincides with $c$. This problem can be expressed by the extensional ESO sentence
 \begin{align*}
     \exists R,G,B \; \forall x,y & \big (
     (R'(x) \Rightarrow R(x)) \wedge 
     (G'(x) \Rightarrow G(x)) \wedge 
     (B'(x) \Rightarrow B(x)) \\
     & \wedge (E(x,y) \Rightarrow \neg (R(x) \wedge R(y)) \wedge \neg(G(x) \wedge G(y)) \wedge \neg(B(x) \wedge B(y))) \\
     & \wedge (R(x) \vee G(x) \vee B(x)) \big ) .
 \end{align*}
 Here, $R$, $G$, and $B$ denote the color classes of the graph, and must denote extensions of input predicates
 $R'$, $G'$, and $B'$ which can be used to
 enforce that some variables are colored by a specific color.
\end{example}

\paragraph{HerFO and the tractability problem.}
As supporting character we have \emph{hereditary first-order logic}. This is a recently
introduced logic that can naturally model several computational problems such as digraph
acyclicity, graph chordality, and
several infinite-domain CSPs~\cite{ManuelSantiagoCSL} --- parametrized variants of this
problem were previously studied, e.g., in~\cite{banachMFCS24,fominTCS64}.
A class $\mathcal C$ of structures is in hereditary first-order logic ($\HerFO$) if
there exists a first-order sentence $\phi$ such that a structure $\bA$ is in ${\mathcal C}$ if
and only if all (non-empty) substructures of $\bA$ satisfy $\phi$. HerFO can be naturally
defined as a fragment of universal monadic second-order logic~\cite[Observation~4.1]{ManuelSantiagoCSL}.

The \emph{tractability problem} for second-order logic asks whether a given SO sentence $\Phi$
describes a polynomial-time solvable problem. The tractability problem is already undecidable
for ESO (and thus, for USO as well); see, e.g.,~\cite{Book}. As a rule of thumb, we expect that if a fragment
$\mathcal L$ of ESO (or of USO) has a rich computational power,  e.g, is NP-rich,
then the tractability problem is  undecidable for $\mathcal L$.
It is still open whether HerFO captures coNP-intermediate problems, nonetheless, it is known that the
tractability problem is already undecidable for HerFO~\cite{ManuelSantiagoCSL}.

\paragraph{Finitely bounded CSPs.}
Our lead secondary character if the family of CSPs of finitely  bounded structures. 
A structure $\bB$ is called \emph{finitely bounded} if there exists a finite set ${\mathcal F}$ of finite
structures such that a finite structure $\bA$ embeds into $\bB$ if and only if no structure from ${\mathcal F}$
embeds into $\bA$. In particular, every finite structure is finitely bounded, but also the countable homogeneous
linear order $(\mathbb Q, <)$ is finitely bounded, as well as the generic partial order and the generic semilinear order~\cite{macphersonDM311}.
Examples from graph theory include the universal homogeneous $K_k$-free
graphs~\cite{LachlanWoodrow} and the generic circular triangle-free graph~\cite{bodirskyJGT109}.

CSPs of finitely bounded structures are closely related to the previously introduced logic SNP. 
In particular, there is a syntactic fragment of SNP that captures CSPs of \emph{reducts} of finitely 
bounded structures~\cite[Corollary~1.4.12]{Book} --- a reduct of a structure $\bA$ is a structure
obtained $\bA'$ by forgetting some relations in $\bA$. Actually, this family of CSPs
(and the corresponding fragment of SNP) has the full computational power of NP~\cite[Proposition~13.2.2]{Book}.
This has lead some people (including us) to believe that CSPs of finitely bounded structures
also capture NP-intermediate problems. It would be an amazing surprise if CSPs of
finitely bounded structures exhibit a P vs.\ NP-complete dichotomy, but as soon as one considered their
reducts, we obtain the full computational power of NP!

\begin{figure}[ht!]
\begin{center}
            \begin{tikzpicture}[scale = 0.6];
        \begin{scope}[scale=0.6]
                \node at (-3.5,12) {\color{blue} \scriptsize NP-rich};
                \node at (-3.5,11) {\color{blue} \textit{\tiny Feder-Vardi~\cite{FederVardi}}};
                \node (ESO) at (24,11) {\scriptsize ESO};
                \node (NP) at (31,11) {\scriptsize NP};
                \node (SNP) at (16,8) {\scriptsize SNP};
                \node  (MESO) at (28,5) {\scriptsize EMSO};
                \node (MSNP) at (8,5) {\scriptsize monotone SNP};
                \node (CSP-SNP) at (0,2) {\tiny CSPs of reducts};
                \node (CSP-SNPa) at (0,1) {\tiny{of finitely bounded structures}};

                \node (extESO) at (24,-7) {\scriptsize ext.\ ESO};
                
                 \draw[teal, dashed, rounded corners] (-7,0) rectangle (34,13);
                \draw[->] (SNP) to  (ESO);
                \draw[->] (MSNP) to  (SNP);
                \draw[->] (MESO) to  (ESO);
                \draw[->] (CSP-SNP) to  (MSNP);
                \draw[->] (extESO) to  (ESO);

                \node at (27.5,11) {\tiny\textit{Fagin~\cite{Fagin}}};
                \node at (12,10.5) {\color{blue} \tiny\textit{Feder-Vardi}};
                \node[rectangle, fill = white, minimum height = 0.25cm, minimum width =2cm] at (22,7){}; 
                \node at (23,7) {\color{blue} \tiny\textit{Feder-Vardi}};
                \draw[->] (ESO) to [bend right]  (NP);
                \draw[->] (NP) to [bend right]  (ESO);
                \node at (4, 6.7) {\color{blue} \tiny\textit{\cite[Cor.~13.2.3]{Book}}};
                \node at (5.4,2.75) {\tiny\textit{\cite[Cor.~1.4.12]{Book}}};

                \draw[orange, dashed, rounded corners] (-7,-5) rectangle (34,-1);
                \node at (-3.5,-2) {\color{orange} \scriptsize Not NP-rich};
                \node at (-3.5,-3) {\color{orange} \tiny (unless E = NE)};
                \node at (-3.5,-4) {\color{orange} \textit{\tiny Theorem~\ref{thm:d-self-reducible-not-NP-rich}}};
                \node at (4,-2) {\color{orange} \scriptsize No dichotomy};
                \node at (4,-3) {\color{orange} \tiny \textit{Theorem~\ref{thm:non-dicho}}};
                \node (d-self) at (31,-2.5) {\scriptsize disjunctive};
                \node (d-selfa) at (31,-3.5) {\scriptsize self-reducible};
                \draw[->] (d-self) to  (NP);
                \node at (28.5,2) {\tiny\textit{\cite[Theorem 1]{koJCSS26}}};
                
                \draw[red, dashed, rounded corners] (-7,-17) rectangle (34,-6);
                \node at (-3.5,-7) {\color{red} \scriptsize Not NP-rich};
                \node at (-3.5,-8) {\color{red} \tiny (unless E = NE)};
                 \node at (-3.5,-9) {\color{red} \textit{\tiny Corollary~\ref{cor:notNPrich}}};

                \node (coHerFO) at (31,-10.5) {\scriptsize  co-HerFO};
                \draw[->, dashed, blue] (ESO) to  [bend right] (MSNP);
                \draw[->, dashed, blue] (ESO) to  [bend right] (MESO);
                \draw[->, dashed, blue] (MSNP) to  [bend right, out = 300] (CSP-SNP);

                 \node (extSNP) at (16,-9) {\scriptsize ext.\ SNP};
                \node  (extMESO) at (26,-14) {\scriptsize  ext.\ EMSO};
                \node  (NPI) at (20,-13) {\scriptsize  $\{$GI, MD$\}$};
                 \node (extMSNP) at (8,-12) {\scriptsize monotone ext.\ SNP};
                 \node (extCSP) at (0,-15) {\tiny CSPs of};
                \node at (0,-16) {\tiny{finitely bounded structures}};
                
                \draw[->] (extCSP) to  (CSP-SNPa);
                \draw[->] (extMSNP) to  (MSNP);
                \draw[->] (extSNP) to  (SNP);
                \draw[->] (extESO) to  (d-selfa);
                \draw[->] (extSNP) to  (extESO);
                \draw[->] (extMSNP) to  (extSNP);
                \draw[->] (extMESO) to  (extESO);
                \draw[->] (extCSP) to  (extMSNP);
                \draw[->, dashed, red] (extMSNP) to  [bend right, out = 300] (extCSP);
                \draw[->, dashed, red] (extESO) to  [bend left = 20] (coHerFO);
                \draw[->, dashed, red] (coHerFO) to  [bend left = 10] (extMESO);
                \draw[->, dashed, red] (NPI) to  (extESO);

                \node at (6,-14) {\tiny\textit{Corollary~\ref{cor:FBstructures}}};
                \node at (4.5,-10.3) {\color{red} \tiny\textit{Theorem~\ref{thm:muSNP-fbCSP-sucSNPwe}}};
                \node at (27,-9) {\color{red} \tiny\textit{Theorem~\ref{thm:meESO-extESO-coHerFO}}};
                \node at (31.3,-12.5) {\color{red} \tiny\textit{Lemma~\ref{lem:HerFO-UMSO}}};
                \node at (29.3,-5.5) {\tiny\textit{Lemma~\ref{lem:extESO->d-self-reducible}}};

                \node at (20,-10.25) {\color{red} \tiny\textit{Section~\ref{sec:NP-intermediate}}};

            \end{scope}

        \end{tikzpicture}
        \end{center}
\caption{\small Classes of computational problems studied in this paper. Solid arcs indicate 
inclusions, and a dashed arc
from $A$ to $B$ indicates that every problem in $A$ is polynomial-time equivalent to a problem in $B$
--- for a cleaner picture we omit the solid arc from  extensional EMSO to  EMSO.
GI stands for the Graph Isomorphism problem, and MD for the Monotone Dualization problem.}
\label{fig:landscape}
\end{figure}

\subsection*{Our contributions}
Our first main result is a characterization of the computational power of extensional ESO in terms of 
hereditary first-order logic (Theorem~\ref{thm:meESO-extESO-coHerFO}): we show that 
for every sentence $\Psi$ in extensional ESO there exists a first-order formula $\phi$ such that 
$\HER(\phi)$ is polynomial-time (in fact, log-space) equivalent to the problem described by the negation of $\Psi$,
and vice versa. As a byproduct of this connection, we get that the tractability
problem for extensional ESO is undecidable (Corollary~\ref{cor:tractability-problem}). 
We then look at fragments of extensional ESO defined by restrictions
on the quantification and on the polarity of atoms of the first-order part. In particular, our second
main result characterizes the computational power of \emph{monotone extensional SNP} (introduced below)
in terms of  CSPs of finitely bounded structures (Theorem~\ref{thm:muSNP-fbCSP-sucSNPwe}).
We also show that extensional ESO captures surjective CSPs, as well as several
graph sandwich problems, orientation completion problems, and constraint topological sorting problems.
We further present examples of candidate NP-intermediate problems expressible in extensional ESO;
namely, Graph Isomorphism and Monotone Dualization (Section~\ref{sec:NP-intermediate}). 

These results point towards a rich computational power of extensional ESO, and we thus
conjecture that extensional ESO does not exhibit a P vs.\ NP-complete dichotomy (Conjecture~\ref{conj:intermediate}). 
However, our last main
result shows that extensional ESO cannot be NP-rich (Corollary~\ref{cor:notNPrich}): there are
problems in $\NP$ that are not polynomial-time equivalent to a problem in extensional ESO, unless E = NE
(the class E is the class of problems computable in exponential time with linear exponent; NE is
its non-deterministic pendant). The assumption that E $\neq$ NE is stronger than the usual assumption
that P $\neq$ NP, but still a common assumption in complexity theory~\cite{BellareGoldwasser,hemaspaandraJCSS5,selmanIC78}.
To prove this result, we show that every problem in extensional ESO is \emph{disjunctively self-reducible},
and that this class is not NP-rich, unless E = NE (Theorem~\ref{thm:d-self-reducible-not-NP-rich});
we complement this result by showing that disjunctive self-reducible sets do not exhibit a P versus
NP-complete complexity dichotomy (Theorem~\ref{thm:non-dicho}). It also follows from this
result that CSPs of finitely bounded structures do not capture the full computational
power of NP, and that HerFO cannot be coNP-rich.

\paragraph{Outline.} 
We begin by recalling some basic concepts from finite model theory (Section~\ref{sect:prelims}) 
and continue to present several families of examples captured by
extensional ESO in Section~\ref{sec:extensional-ESO}.
In Section~\ref{sect:uSNP} we show that extensional ESO and HerFO have the same computational power. 
An important restriction of extensional ESO is monotone extensional SNP (Section~\ref{sect:mon-ext-SNP}), which we show has the
same computational power as  CSPs of finitely bounded structures (Theorem~\ref{thm:muSNP-fbCSP-sucSNPwe}). 
In Section~\ref{sec:NP-intermediate} we prove that Graph Isomorphism 
and Monotone Dualization are captured by extensional ESO (up-to polynomial-time equivalence).
We discuss the limitations of the computational power of extensional ESO in Section~\ref{sect:limitations}, and in particular
we prove that extensional ESO is not NP-rich, unless E = NE (Corollary~\ref{cor:notNPrich}).
We close with a series of problems that are left open (Section~\ref{sect:open}).

\section{Preliminaries} 
\label{sect:prelims} 
We assume basic familiarity with first-order logic and we follow standard notation from
model theory, as, e.g., in~\cite{Hodges}. We also use standard notions from complexity theory.

\subsection{(First-order) structures}
Given a relational signature $\tau$ and a $\tau$-structure $\bA$,
we denote by $R^\bA$ the interpretation in $\bA$ of a relation symbol $R\in \tau$. Also, 
we denote relational structures with letters $\bA, \bB, \bC,\dots$, and their domains by
$A,B,C,\dots$. In this article, structures have non-empty domains. 

If $\bA$ and $\bB$ are $\tau$-structures,
then a \emph{homomorphism} from $\bA$ to $\bB$ is a map $f \colon A \to B$ such that
for all $a_1,\dots,a_k \in A$ and $R \in \tau$ of arity $k$, if $(a_1,\dots,a_k) \in R^{\bA}$, then 
$R(f(a_1),\dots,f(a_k)) \in R^{\bB}$ 
We write $\bA \to \bB$ if there exists a homomorphism from $\bA$ to $\bB$, and we denote by 
$\CSP(\bB)$ the class of finite structures $\bA$ such that $\bA\to \bB$. 
Let ${\mathcal C}$ be a class of finite $\tau$-structures.
We say that ${\mathcal C}$ is 
\begin{itemize}
    \item \emph{closed under homomorphisms} if
    for every $\bB \in {\mathcal C}$, if $\bB \to \bA$, then $\bA \in {\mathcal C}$ as well.
    \item \emph{closed under inverse homomorphisms} if for every $\bB \in {\mathcal C}$, if $\bA \to \bB$, then $\bA \in {\mathcal C}$ as well (i.e., the complement of ${\mathcal C}$ in the class of all finite $\tau$-structures is closed under homomorphisms). 
\end{itemize}
Note that $\CSP(\bB)$ is closed under inverse homomorphisms. 

If $\bA$ and $\bB$ are $\tau$-structures with disjoint domains $A$ and $B$,
respectively, then the \emph{disjoint union} $\bA \uplus \bB$ is the $\tau$-structure $\bC$ with domain
$A \cup B$ and the relation $R^{\bC} = R^{\bA} \cup R^{\bB}$ for every $R \in \tau$. Note that
$\CSP(\bB)$ is \emph{closed under disjoint unions}, i.e., if $\bA \in {\mathcal C}$ and
$\bB \in {\mathcal C}$, then $\bA \uplus \bB \in {\mathcal C}$. 

For examples we will often use graphs and digraphs. Since this paper is deeply involved
with logic and finite model theory, we will think of digraphs as binary structures with signature
$\{E\}$, following and adapting standard notions from graph theory~\cite{bondy2008} to this setting. 
In particular, given a digraph $\bD$, we call $E^\bD$ the \emph{edge set} of $\bD$, and its
elements we call \emph{edges} or \emph{arcs}.
Also, a \emph{graph} is a digraph whose
edge set is a  symmetric relation, and an \emph{oriented graph} is a digraph whose edge
set is an anti-symmetric relation. Given a positive integer $n$, we denote by 
$K_n$ the complete graph on $n$ vertices, i.e., the graph with vertex set $[n]$
and edge set $(i,j)$ where $i\neq j$. A \emph{tournament} is an oriented graph whose symmetric
closure is a complete graph,  and an \emph{oriented path} is an oriented graph whose symmetric
closure is a path. We denote by $\vec{P_n}$ the \emph{directed path} on $n$ vertices, 
 i.e., the oriented path with vertex set $[n]$ and edges $(i,i+1)$ for $i\in[n-1]$.

\subsection{SNP and its fragments}
Let $\tau$ be a finite set of relation symbols.
An \emph{SNP $\tau$-sentence} (short for \emph{strict non-deterministic polynomial-time})
is a sentence of the form 
$$ \exists R_1,\dots,R_k \forall x_1,\dots,x_n. \psi$$
where $\psi$ is a quantifier free $\tau \cup \{R_1,\dots,R_k\}$-formula. 
If the structure $\bA$ satisfies the sentence $\Psi$, we write $\bA \models \Psi$.
We say that a class of finite $\tau$-structures $\mathcal C$ is in \emph{SNP} if
there exists an SNP $\tau$-sentence $\Phi$ such that  $\bA \models \Phi$ if and only
if $\bA \in {\mathcal C}$.  Note that the class of (finite) models of an SNP sentence is
closed under substructures.

A first-order $\tau$-formula $\phi$ is called \emph{positive} if it does not use the negation symbol $\neg$
(so it just uses the logical symbols $\forall,\exists,\wedge,\vee,=$, variables and symbols from $\tau$). 
The negation of a positive formula is called \emph{negative}; note  that every negative formula is equivalent
to a formula in prenex conjunctive normal form where every atomic formula appears in negated form, and all
negation symbols are in front of atomic formulas; such formulas will also be called negative.

Consider a first-order sentence  $\phi$ in prenex conjunctive normal form without (positive)
literals of the form $x = y$. Let $\psi$ be a clause of $\phi$, and let 
$\bC$ be the structure whose vertex set is the set of variables appearing in $\psi$, and
the interpretation of $R\in \tau$ consists of those tuples $\overline x$ such that 
$\psi$ contains the negative literal $\lnot R(\overline x)$
(this structure is sometimes called the \emph{canonical database} of $\psi$). We say $\psi$ is \emph{connected}
if $\bC$ is connected, and that $\phi$ is \emph{connected} if every clause of $\phi$ is connected --- notice that
in particular, the notion of connectedness is well-defined for negative first-order logic.
Given an $\SNP$ $\tau$-sentence $\Phi$ with first-order part $\phi$ in conjunctive normal form (CNF), we say
that $\Phi$ is 
\begin{itemize}
    \item \emph{connected} if $\phi$ is connected (so $\phi$ does not contain (positive) literals of the form $x = y$), 
    \item \emph{monotone} if every symbol from $\tau \cup \{=\}$ is negative in $\phi$, 
    \item \emph{monadic} if every existentially quantified symbol is monadic (i.e., unary), and
    \item \emph{without equalities} if $\phi$ does not use the equality symbol (so it just uses the logical symbols $\forall,\exists,\wedge,\vee,\lnot $, variables and symbols from $\tau$).
\end{itemize}
Notice that for every SNP sentence $\Phi$ we may assume that it does not
contain literals of the form $\lnot (x =  y)$; otherwise, we obtain an equivalent formula by deleting every such literal
and replacing each appearance of the variable $y$ in the same clause by the variable $x$. Hence, in this paper we assume without loss of generality that if $\Phi$ is in monotone $\SNP$, then it is in $\SNP$ without equalities.\footnote{We deviate from the notation from~\cite{FederVardi}:
they call this fragment $\SNP$ without inequalities since they assume negation normal form and we do not.
} Connected monotone $\SNP$ is closely related to $\CSP$s in the following sense.

\begin{theorem}
[Corollary 1.4.12 in~\cite{Book}]
\label{thm:SNP-CSP}
    An $\SNP$ sentence $\Phi$ describes a problem of the form $\CSP(\bA)$ for some (infinite) structure
    $\bA$ if and only if $\Phi$ is equivalent to a connected monotone $\SNP$ sentence.
\end{theorem}

\subsection{Polynomial-time reductions and subclasses of NP}
We consider
three kinds of polynomial-time reductions which we specify now. We say that a set
$L$ \emph{reduces in polynomial time} to a set $L'$ if there is a polynomial-time
computable function $f$ such that $x\in L$ if and only if $f(x)\in L'$; we also say that $f$ is a
\emph{polynomial-time many-one reduction}. More generally, 
we say that there is a \emph{disjunctive polynomial-time reduction} from $L$ to $L'$ if
there is a polynomial-time computable function $g$ which computes 
 a set of queries $g(x)$ such that $x\in L$ if and only if $y\in L'$ for some 
 $y\in g(x)$.
Similarly, if $g$ satisfies that $x\in L$ if and only if $y\in L'$ for
all $y\in g(x)$, 
we say that there is a \emph{conjunctive polynomial-time reduction} from $L$ to $L'$.
Finally, we say that there is a \emph{polynomial-time Turing-reduction} from $L$ to $L'$,
if there is a polynomial-time deterministic oracle machine $M$ that accepts $L$ using
$L'$ as an oracle. Clearly, if there exists a polynomial-time reduction from $L$ to $L'$,
then there are also disjunctive and conjunctive polynomial-time reductions, and, in turn,
if there exists a disjunctive or a conjunctive polynomial-time reduction from $L$ to $L'$,
then there is also a Turing-reduction from $L$ to $L'$. We say that $L$ and $L'$ are \emph{polynomial-time equivalent}
(\emph{disjunctive-}, \emph{conjunctive}-, and \emph{Turing-equivalent}, respectively) if 
there are polynomial-time reductions (disjunctive-, conjunctive-, and Turing-reductions, respectively)
from $L$ to $L'$ and from $L'$ to $L$.
All of these notions of reducibility also exist in the variant with \emph{log-space} instead of \emph{polynomial-time}; in this case we require the existence of a Turing machine which has access to a (read-only) input tape, a (write-only) output tape, and a work tape whose size is logarithmic in the size of the input. Note that any function which is log-space computable in this sense is also polynomial-time computable.

\section{Extensional ESO}
\label{sec:extensional-ESO}
The aim of this section is to present several families of computational
problems considered in the literature that are captured by extensional ESO.
First, we recall the definition of extensional ESO.
An ESO $\tau$-sentence $\Psi$ is called \emph{extensional} if for every existentially quantified
predicate $R$ there is a distinct symbol $R'\in \tau$ of the same arity such that $\Psi$
contains a conjunct $\forall \overline{x} \big (R'(\bar x) \Rightarrow R(\bar x) \big )$, and $R'$ does
not appear anywhere else in $\Psi$. In this case we say that $R$ \emph{extends} $R'$.

Actually, one can reinterpret extensional ESO as follows. Consider a
signature $\sigma$ which is the disjoint union of two signatures $\tau_1$ and $\tau_2$, 
and let $\phi$ be a first-order $\sigma$-sentence. 
On an input $\sigma$-structure, 
the task is to decide whether there is an extension $\bA'$ of $\bA$ where
we are not allowed to modify the interpretation of symbols $R\in \tau_1$, i.e., 
$A' = A$, $R(\bA') = R(\bA)$ for each $R\in \tau_1$, and $R(\bA) \subseteq R(\bA')$
for each $R\in \tau_2$. Clearly, these problems can be expressed with extensional ESO sentences.

\subsection{Surjective CSPs}
\label{ex:finite-CSPs}
Arguably, the most natural ESO sentence describing $k$-colourability has $k$ unary quantified
predicates $U_1,\dots, U_k$, and the first-order part states that every vertex belongs to exactly one $U_i$, 
and that each $U_i$ is an independent set. Such a sentence is not in extensional ESO, but  $k$-colourability
can be expressed in extensional ESO as follows.

\begin{example}[$k$-colouring problem]\label{ex:Kn-USNP}
    Let $E_1$  be a binary relational symbol. For every positive integer $k$, denote
    by $\phi_k$ the universal $\{E_1\}$-formula stating that $(V;E_1)$ is a loopless symmetric graph
    with no induced $K_1+K_2$ (the three-element graph which is formed as the disjoint union of an
    isolated vertex and an edge) and no $K_{k+1}$. Hence, $(V;E_1)\models \phi_k$ if and 
    only if $(V;E_1)$ is a \emph{complete $k$-partite} graph, i.e., there
    is partition $P$ of $V$ with at most $k$ parts such that there is an edge $uv \in E_1$ if and only
    if $u$ and $v$ lie in different parts of $P$. Now, consider the extensional ESO sentence
    \[
        \Phi:=\exists E_1\, \big (\phi_k \wedge \forall x,y\, (E(x,y)\Rightarrow E_1(x,y)) \big)
    \]
    Then a graph $(V;E)$ satisfies $\Phi$ if and only if there is an 
    edge set $E_1 \supseteq E$ such that $(V;E_1)$ is a complete $k$-partite graph. It thus follows
    that $(V; E)\models \Phi$ if and only if $(V; E)$ is $k$-colourable.
\end{example}

Let $\bA$ and $\bB$ be relational $\tau$-structures.
A \emph{full-homomorphism} is a homomorphism $f\colon \bA\to \bB$ such that for
every $r$-ary relation symbol $R$ and every $r$-tuple $\bar a$ of $A$ it is the
case that $\bar a \in R^\bA$ if and only if $f(\bar a)\in R^\bB$. In the following, 
will assume that $\tau$ is finite. We denote by 
$\CSP_{\full}(\bA)$ the class of finite structure $\bA$ that admit a full homomorphism to $\bB$. 

It is straightforward to observe that  $\bA$ maps
homomorphically to $\bB$ if and only
if there is an extension $\bA'$ of $\bA$ that admits a full-homomorphism to 
$\bB$. It was proved in~\cite{ballEJC31} that for every finite $\tau$-structure
$\bB$ there is a finite set $\calF$ of finite $\tau$-structures such that $\CSP_{\full}(\bB)$ is the class
of $\calF$-free finite structures. In particular, there is a first-order sentence $\phi$
such that $\CSP_{\full}(\bB)$ coincides with the finite models of $\phi$. Hence,
if $\tau = \{R_1,\dots, R_k\}$ is the signature of $\bB$, and $\phi'$ is obtained from $\phi$ by replacing
each symbol $R_i$ by $R_i'$, the CSP of $\bB$ is described by the extensional 
ESO sentence
\[
    \exists R_1',\dots, R_k' \big ( \bigwedge_{R\in \tau}  \forall \overline x \;
    (R(\bar x)\Rightarrow R'(\bar x)) \land \phi' \big ).
\]

\emph{Surjective CSPs} are the variant of CSPs where the task is to decide
whether there exists a surjective homomorphism $f\colon\bA\to\bB$ from a given 
input structure $\bA$ to a fixed template structure $\bB$. 
Similarly as
for CSPs one can show that for every finite structure $\bB$ the surjective CSP for $\bB$ is
in extensional ESO as follows. 
For a finite $\tau$-structure $\bB$, let $\delta_\bB$ denote the \emph{diagram}
of $\bB$, that is, the conjunctive quantifier-free formula whose variables 
$x_b$ are 
indexed by elements $b\in B$ and such that for each symbol $R\in \tau$
of arity $r$, the formula $\delta_{\bB}$ contains the conjunct $R(x_{b_1},\dots, x_{b_r})$ if 
$(b_1,\dots, b_r)\in R(\bB)$, and contains the conjunct
$\lnot R(x_{b_1},\dots, x_{b_r})$ otherwise. We denote by $\delta'_\bB$ the existential 
sentence obtained from the diagram $\delta_\bB$ by existentially quantifying each
variable and replacing each occurrence of $R \in \tau$ by $R'$. Now, using the arguments above, it follows that there is a surjective
homomorphism form $\bA$ to $\bB$ if and only if $\bA$ satisfies the extensional
ESO sentence
\[
    \exists R_1',\dots, R_k' \big (\bigwedge_{R\in \tau} \forall \overline x (
    R(\bar x)\Rightarrow R'(\bar x)) \land \phi' \land \delta'_\bB \big ).
\]

\subsection{Graph Sandwich Problems}
Graph Sandwich Problems were introduced by Golumbic, Kaplan, and Shamir in~\cite{golumbicJA19}. 
The \emph{sandwich problem} (SP) for a graph class ${\mathcal C}$ is the following computational problem.
The input is a pair of graphs $((V,E_1),(V,E_2))$ where $E_1\subseteq E_2$, and the task
is to decide whether there is an edge set $E$ such that $E_1\subseteq E\subseteq E_2$ and the
graph $(V,E)$ belongs to $\calC$. Equivalently, the SP for $\calC$ can receives as an
input a triple $(V,E,N)$ where $V$ is a set of vertices, $E$ is a set of edges, and $N$ is a set of non-edges.
The task is to determine whether there is a graph $(V,E') \in \calC$ such that $E\subseteq E'$ and 
$E'\cap N = \varnothing$.

Considerable attention has been placed on the SP for $\mathcal F$-free graphs for 
finite sets of graphs $\mathcal F$~\cite{alvaradoAOR280,dantasDAM159,
dantasENTCS346,figueiredoDAM251}. In particular, for every such finite
set $\mathcal F$, there is a universal sentence $\phi$ such that a $\{E\}$-structure $\bG$
satisfies $\phi$ if and only if $\bG$ is an $\mathcal F$-free graph.
Consider the binary signature $\{E,N\}$, and let $\phi'$ be the universal sentence
obtained by replacing each occurrence of a $E$ by $E'$, and a every occurrence of 
$\lnot E$ by $N'$. It is now straightforward to verify  that the SP for $\mathcal F$-free
graphs is described by the extensional ESO sentence
\begin{align*}
        \exists E', N'\; \forall x,y \big ( & (E(x,y)\Rightarrow E'(x,y)) \land (N(x,y)\Rightarrow N'(x,y)) \land (N'(x,y) \Leftrightarrow \lnot E'(x,y)) \land \phi' \big ).
\end{align*}

\subsection{Orientation Completion Problems}
An \emph{oriented graph} is a digraph $\bD = (V,A)$ with no symmetric
pair of edges, i.e., there are no $x,y \in V$ such that $(x,y) \in A$ and $(y,x) \in A$.
Orientation completion problems were introduced by Bang-Jensen,
Huang, and Zhu~\cite{bangjensenJGT87}.
Given a finite set of oriented graphs $\mathcal F$, the orientation 
completion to $\mathcal F$-free graphs is the following. The input
is a partially oriented graph $(\bG, A')$, i.e., a graph $\bG$ together with a set of edges  $A'\subseteq E(\bG)$
such that $(V(\bG),A')$ is an oriented graph. The task is to decide whether there exists an \emph{$\mathcal F$-free orientation completion of $(\bG, A)$}, i.e.,
an edge set $A$ such that 
\begin{itemize}
    \item $A'\subseteq A$, 
    \item $(G,A)$ is an oriented
graph that does not embed any $\bF\in \mathcal F$, and 
\item the symmetric
closure of $A$ equals $E(\bG)$. 
\end{itemize}
Clearly, for every finite set of oriented
graphs $\mathcal F$ there is a universal $\{A\}$-sentence $\phi$ that describes
the class of $\mathcal F$-free oriented graphs. Hence, the $\mathcal F$-free
orientation problem is captured by
\[
    \Phi:= \exists A\; \forall x,y \big ( (A'(x,y)\Rightarrow A(x,y)) ~\land~ ((A(x,y)\lor A(y,x)) \Leftrightarrow E(x,y)) ~\land~\phi \big ).
\]

\subsection{Constraint Topological Sorting}
A \emph{topological sort} of a directed acyclic graph $\bD$ is a linear ordering
$<$ of $D$ such that if $(x,y)\in E(\bD)$, then $x < y$.
\emph{Constraint topological sorting problems (CTSs)} arise in the context
of automata theory  and regular languages~\cite{amarilliICALP2018}. 
These problems are described
by a regular language $L$; the input is a directed acyclic graph $\bD$
whose vertices are labeled with letters of some alphabet $\Sigma$ and the task
is to decide whether there exists a topological sort of $\bD$ where the word obtained by reading the label of the vertices according
to the order $<$ belongs to the language $L$. Amarilli and Paperman conjecture
that CTS for regular languages exhibit an NL vs.\ NP-complete dichotomy~\cite[Conjecture~2.3]{amarilliICALP2018}. 
By considering a slight modification on the problem, they show that
such a dichotomy holds for aperiodic regular languages~\cite[Theorem~5.2]{amarilliICALP2018}. 
Aperiodic regular languages correspond to regular languages expressible
in first-order logic~\cite{straubingSIGLOG5} where the signature consists of a unary relation symbol
for each letter of the alphabet $\Sigma$ and a binary relation symbol $s$ encoding
the successor relation. Clearly, the successor relation can be first-order defined 
from the order relation, e.g., 
\[
s(x,y):= (x <y)\land (\forall z.\; (x<z<y \Rightarrow (z =x \lor z = y)).
\]
Hence, for every aperiodic regular language $L$ there is a first-order
$\Sigma~\cup \{<\}$-sentence $\phi$ that describes the language $L$. 
Therefore,  the CTS for $L$ is expressible in extensional ESO by
\[
    \exists{<}\; \big ( \forall x,y (E(x,y)\Rightarrow x < y) \land \phi \big ).
\]

\section{Extensional ESO and HerFO}
\label{sect:uSNP}
    Every problem in $\HerFO$ is polynomial-time equivalent to the complement
    of a problem described by monadic extensional ESO. This follows from 
    the following observation~\cite{ManuelSantiagoCSL}.
     
\begin{observation}\label{obs:FO-relative}
Consider a first-order $\tau$-formula $\phi:=Q_1x_1\dots Q_nx_n. \psi(x_1,\dots, x_n)$.
If $\psi(x)$ is a first-order $\tau$-formula, then there is a $\tau$-formula $\phi_\psi$
such that a $\tau$-structure $\bA$ satisfies $\phi_\psi$ if and only if the
substructure of $\bA$ with domain $\{a\in  A\colon \bA\models \psi(a)\}$ satisfies $\phi$. Namely, 
    \[
   \phi_\psi := \;  Q_1x_1\dots Q_nx_n \left ( \bigwedge_{i\in U}\psi(x_i) \implies \bigwedge_{i\in E}\psi(x_i)\land \phi(x_1,\dots, x_n) \right),
   \]
where $U$ (respectively, $E$) is the set of indices $i\in \{1,\dots,n\}$ such that $Q_i$ is a universal
(respectively, existential) quantifier. 
\end{observation}

We will use the formula $\phi_\psi$ from Observation~\ref{obs:FO-relative} in the following lemma
and in the proof of Theorem~\ref{thm:extESO->HerFO}.

\begin{lemma}\label{lem:HerFO-UMSO}
   Consider a first-order $\tau$-formula $\phi:=Q_1x_1\dots Q_nx_n. \psi(x_1,\dots, x_n)$. 
   Then, $\HER(\phi)$ and the complement of the problem described by the monadic extensional
    $(\tau\cup\{S\})$-sentence $\Psi$ are log-space equivalent, where 
   \[
   \Psi:= \exists \bar{S}. \;  Q_1x_1\dots Q_nx_n \left (\forall y (S(y) \Rightarrow \bar{S}(y)) \land  \bigwedge_{i\in E} \lnot \bar{S}(x_i) \land \big (\bigwedge_{i\in U} \lnot \bar{S}(x_i) \Rightarrow \psi(x_1,\dots, x_n) \big ) \right).
   \]
\end{lemma}
\begin{proof}
    First notice that if an expansion $\bA' := (\bA, \bar{S})$ satisfies the first-order part of $\Phi$
    and $\psi(x):= \lnot \bar{S}(x)$,
    then $(\bA, \bar{S})\models \phi_\psi$ (see Observation~\ref{obs:FO-relative} for the definition
    of $\phi_\psi$).
    Moreover, $(\bA, \bar{S})$ satisfies the first-order part of $\Phi$ if
    and only if $S^{\bA} \subseteq  \bar{S}^{\bA'}$, $A \not\subseteq \bar{S}^{\bA'}$, and $\bA'\models \phi_\psi$.
    Hence, the log-space reduction from $\HER(\phi)$ to $\lnot \Psi$ is straightforward:
    every $\tau$-structure $\bA$ is a $(\tau\cup\{S\})$-structure, and it thus follows via
    Observation~\ref{obs:FO-relative} and the definition of $\phi_\psi$ that $\bA\models \lnot\Psi$
    if and only if every non-empty substructure of $\bA$ satisfies $\phi$, 
    i.e., if and only if $\bA \in \HER(\phi)$. The converse reduction follows with similar arguments. Indeed, if $\bA$
    is a $(\tau\cup\{S\})$-structure,  then $\bA\models \lnot\Psi$ if and only if $S^\bA \neq A$ and 
    the substructure $\bB$ of $\bA$ with domain $A\setminus S^\bA$ models $\lnot\Psi$. Similarly as above,
    $\bB\models \lnot \Psi$ if and only if $\bB$ hereditarily satisfies $\phi$. Clearly, $\bB$ can be computed
    from $\bA$ in polynomial-time, and thus $\lnot\Psi$ reduces in polynomial time (and in logarithmic space)
    to $\HER(\phi)$. 
\end{proof}

Combining this lemma and Corollary~\ref{cor:notNPrich}, we get the following application of our results.

\begin{corollary}\label{cor:HerFO-notcoNPrich}
    $\HerFO$ is not $\coNP$-rich, unless E = NE.
\end{corollary}

Now we use the fact that the tractability problem for HerFO is undecidable~\cite{ManuelSantiagoCSL}
to conclude that the tractability problem for extensional ESO is also undecidable (via Lemma~\ref{lem:HerFO-UMSO}).

\begin{corollary}\label{cor:tractability-problem}
    The tractability problem for (monadic) extensional ESO is undecidable.
\end{corollary}

\subsection{HerFO captures extensional ESO}
\label{sect:usnp-herfo}
Now we see that for every problem in extensional ESO, there is a polynomial-time equivalent
problem in $\HerFO$. That is, for every extensional ESO $\tau$-sentence $\Psi$ 
we propose a first-order $\sigma$-formula $\phi$  such that
$\neg \Psi$ is polynomial-time equivalent to $\HER(\phi)$. First note that $\neg \Psi$
can be viewed as a universal second-order sentence. The idea for one reduction is that
given a $\tau$-structure $\bA$, we construct a $\sigma$-structure $\bB$
such that every expansion of $\bA$ is encoded by some substructure of $\bB$, 
and every substructure of $\bB$ either encodes an expansion of $\bA$, or it trivially
satisfies $\phi$. We illustrate a construction of $\bB$ in Figure~\ref{fig:extESO->HerFO}.
Naturally,  we will define $\phi$ in such a way that a substructure $\bB'$ of $\bB$
encoding an expansion $\bA'$ of $\bA$ satisfies $\phi$ if and 
only if $\bA'$ satisfies the first-order part of $\lnot \Psi$.
Hence, we obtain a polynomial-time reduction from $\lnot\Psi$ to $\HER(\phi)$.
For the remaining reduction we need to know how to handle $\sigma$-structures that are not
in the image of the previous construction. This will also be taken care of by $\phi$: we will include some clauses
in $\phi$ that will allow us to find a substructure $\bB_1$ of $\bB$ such that
$\bB$ hereditarily satisfies $\phi$ if and  only if
$\bB_1$ hereditarily satisfies $\phi$, and $\bB_1$ is constructed (as before) from some
$\tau$-structure $\bA$. This will yield a polynomial-time reduction from $\HER(\phi)$ to $\lnot \Psi$.

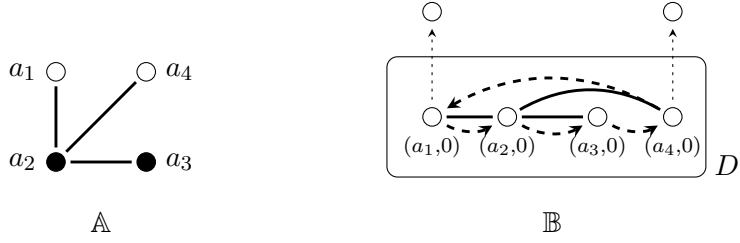
\begin{figure}[ht!]
\centering
\begin{tikzpicture}

  \begin{scope}[xshift = -3cm, scale = 0.8]
    \node [vertex, label = left:$a_1$] (1) at (-.75,.75) {};
    \node [vertex, fill = black, label = left:$a_2$] (2) at (-.75,-.75) {};
    \node [vertex, fill = black,  label = right:$a_3$] (3) at (.75,-.75) {};
    \node [vertex, label = right:$a_4$] (4) at (.75,.75) {};
    \node at (0,-1.75) {$\bA$};

    \foreach \from/\to in {1/2, 2/3, 2/4} 
    \draw [edge] (\from) to (\to);
  \end{scope}

  \begin{scope}[xshift = 3cm, scale = 0.8]
    \node [vertex, label = {below:$\scriptstyle  (a_1,0)$}] (1) at (-2,0) {};
     \node [vertex, label = {below:$\scriptstyle  (a_2,0)$}] (2) at (-0.75,0) {};
     \node [vertex, label = {below:$\scriptstyle  (a_3,0)$}] (3) at (0.75,0) {};
     \node [vertex, label = {below:$\scriptstyle   (a_4,0)$}] (4) at (2,0) {};
     \node at (0,-1.75) {$\bB$};

    \draw[rounded corners] (-2.75,-1) rectangle (2.55,1);
    \node at (2.9,-0.8) {$D$};

    \node [vertex] (s1) at (-2,1.75){};
    \node [vertex] (s4) at (2,1.75){};
      
    \foreach \from/\to in {1/2, 2/3} 
    \draw [edge] (\from) to (\to);
    \foreach \from/\to in {1/s1, 4/s4} 
    \draw [arc, dotted, thin] (\from) to (\to);
    \draw [edge] (2) to [bend left] (4);
     \foreach \from/\to in {1/2, 2/3, 3/4, 4/1} 
    \draw [arc, dashed] (\from) to [bend right] (\to);
  \end{scope}
  
\end{tikzpicture}
\caption{Let $R$ be a binary and $U'$ be a unary relation symbol. Consider an extensional ESO $\{R,U'\}$-formula
$\Phi$ of the form $\exists U \forall x ((U'(x) \Rightarrow U(x)) \land \phi)$
where $\phi$ is a first-order $\{R,U\}$-formula. On the left, we depict an $\{R,U'\}$-structure
where the interpretation of $U'$ is indicated by black vertices, and the interpretation of $R$ is indicated by 
undirected  edges.  On the right, we depict the structure $\bB$
with signature 
$\{R, D, E, <, S_U\}$, where $D^{\bB}$ is indicated by the vertices inside the rectangle, $E^{\bB}$ is indicated by bended dashed directed edges,
$<^{\bB}$ is the strict left-to-right linear order on $A \times \{0\}$, and the interpretation of $S_U$ is indicated by straight dotted arcs.
Notice that every substructure $\bB'$ of $\bB$ either does not contain a directed $E$-cycle, or it encodes an expansion of $\bA$
by a unary predicate $U$ where $U(x)$ if and only if $\bB'\models\lnot \exists y. S_U(x,y)$. 
Moreover, this expansion satisfies the clause
$\forall x (U'(x) \Rightarrow U(x))$. We use these facts in the proof of Theorem~\ref{thm:extESO->HerFO}.}
\label{fig:extESO->HerFO}
\end{figure}

\begin{theorem}\label{thm:extESO->HerFO}
    For every sentence $\Psi$ in extensional $\ESO$ there exists a first-order formula $\phi$ 
    such that the problem described by $\neg \Psi$ is log-space equivalent to $\HER(\phi)$.
\end{theorem}
\begin{proof}
    Let $\tau$ be the input signature of $\Psi$. We may assume that $\Psi$ is of the form 
$$\exists R_1,\dots,R_n \, Q_1 x_1 \dots Q_m x_m. \psi$$ for some quantifier-free prefix $Q = (Q_1 \dots Q_m)$ and quantifier-free formula $\psi$ in conjunctive normal form.
    For each $i\in [n]$ let $R'_i$ be an $r_i$-ary predicate in
    $\tau$ such that $\psi$ contains the clause $R'_i(y_1,\dots, y_{r_i}) \Rightarrow R_i(y_1,\dots, y_{r_i})$
    where each $y_i$ is a universally quantified variable, and the symbol $R'$ does not
    appear in any other clause of $\psi$. We denote by $\tau'$ the subset of $\tau$ consisting of
    $R_i'$ for every $i\in[n]$.
 
    We first consider the trivial cases where either all finite structures satisfy $\Psi$, or no
    finite structure satisfies $\Psi$.  In the former case, we may pick $\phi$ to be the formula
    that is always false, and in the latter, we pick the formula that is always true. 
   Then we have a trivial reduction from $\HER(\phi)$ to $\neg \Psi$ and vice versa. 

    Let $\rho$ be the signature that contains $\tau$, a unary symbol $D$, two binary symbols $E$ and $<$,
    and for each $i\in[n]$ a symbol $S_i$ of arity $r_i+1$ (as we will see below, these symbols $S_i$ will
    encode the existentially quantified symbol $R_i$).
    Let $\phi$ be the formula $(\phi_1 \lor \phi_2) \land \phi_3 \land \phi_4 \land \phi_5 \land \phi_6$ defined as follows.
    \begin{enumerate}
        \item We define $\phi_1$ as follows: first, let $\chi: = \lnot (Q_1 x_1 \dots Q_m x_m. \psi)$ and $\delta(x):= D(x)$; secondly, 
        let $\chi' : = \chi_\delta$ (see
        Observation~\ref{obs:FO-relative}); finally, let $\phi_1$ be the formula obtained from $\chi'$ by replacing occurrences
        of $R_i(y_1,\dots, y_r)$ by $\forall z. \lnot S_i(y_1,\dots, y_r,z)$. It will also be convenient
        to denote by $\phi_1'(x_1,\dots, x_m)$ the formula where the bounded variables $x_1,\dots, x_m$ are
        now free variables (but the universally quantified variables $z$ remain bounded). 
        \item Let $\phi_2$ be the formula stating that there is some $y$ such that for all $x_1,\dots, x_m$
        the formulas $\lnot E(x_1,y)$ and  $\lnot S_i(x_1,\dots, x_{r_i},y)$ hold for all $i\in[n]$.
        \item Let $\phi_3$ be the formula that asserts that if
        $S_i(\overline{x},y)$ holds, then $y$ does not appear in any other tuple of the structure.
        \item Let $\phi_4$ be the formula stating that if $E(x,y)$ holds, then $D(x)$ and $D(y)$ hold. 
        \item Let $\phi_5$ be the formula stating that the in-degree and out-degree of every vertex
        with respect to  the relation $E$ is at most $1$.
        \item Let $\phi_6$ be the formula stating that $<$ is a strict linear order on the subset $D$
        such that there is at most one edge (of $E$) \textit{going backwards with respect to $<$}, i.e.,
        \[
        \forall x,y,z,w\in D \big ( (x < y \land z < w \land E(y,x) \land E(w,z)) \Rightarrow (x = z \land y = w) \big ).
        \]
    \end{enumerate}
    We claim that  $\neg \Psi$ and $\HER(\phi)$ are polynomial-time equivalent.

    For the polynomial-time reduction from $\neg \Psi$ to $\HER(\phi)$ consider a finite $\tau$-structure
    $\bA$ with domain $A = \{a_{0},\dots, a_{\ell-1}\}$. We first construct an auxiliary structure $\bF$
    with domain $$A\times \{0\} \cup A^{r_1}\times \{1\} \cup \cdots \cup A^{r_n}\times \{n\}$$ such that 
    \begin{itemize}
        \item 
        the map $a \mapsto (a,0)$ is an isomorphism between $\bA$ and the $(\tau\setminus \tau')$-reduct
        of the substructure of $\bF$ with domain $A\times\{0\}$,
        \item $D(x)$ holds for every $x\in A\times \{0\}$,
        \item $E((a_{i-1},0),(a_i,0))$ holds for every $i \in [\ell]$ (where indices are modulo $\ell$),
        \item $(a_i,0) < (a_j,0)$ if and only if $i < j$, for every $i,j\in \{0,\dots, \ell-1\}$, 
        \item $S_i((a_{j_1},0),\dots, (a_{j_{r_i}},0),(a_{j_1},\dots, a_{j_{r_i}},i))$ holds for all $i\in [n]$ and all
        $a_{j_1},\dots, a_{j_{r_i}} \in A$, and 
        \item no other relations hold in $\bF$. 
    \end{itemize}
    Now, we consider the substructure $\bB$ of $\bF$ with domain 
    \[
    B:=\{(a,0)\mid a\in A\} \cup \{(\overline{a},i) \mid \bA\models \lnot R_i'(\overline{a})\}.
    \]
    Clearly, $\bB$ can be computed from $\bA$ in polynomial time and logarithmic space ---
    we provide an illustration of this construction 
    in Figure~\ref{fig:extESO->HerFO}.
    We have to argue that $\bA \models \neg \Psi$ if and only if 
    all substructures of $\bB$ satisfy $\phi$. Note that 
    $\bB$ satisfies $\phi_3$, $\phi_4$, $\phi_5$, and $\phi_6$ by construction.
    Since $\phi_3 \land \phi_4\land\phi_5 \land \phi_6$ is a universal
    sentence,  all substructures of $\bB$ satisfy $\phi_3\land \phi_4 \land \phi_5 \land \phi_6$.
    Now, suppose that a  substructure $\bB'$ of $\bB$ does not contain $(a_i,0)$ for some
    $i \in \{0,\dots,\ell-1\}$. First consider the case that  $(a_j,0)\in B'$ for some $j\in\{0,\dots, \ell-1\}$.
    It follows that $\bB'$ satisfies $\phi_2$,
    because the vertex $(a_k,0)$ where $k$ is the first index $i'$
    from $i,i+1,\dots, j$ (where indices are modulo $\ell$) such that $(a_{i'},0)$ belongs to the domain of
    $\mathbb B'$  provides a witness for the existentially quantified variable $y$ in $\phi_2$ (the conjuncts
    $\lnot S_i(x_1,\dots, x_{r_i},y)$ are satisfied by $(a_k,0)$ because no element from $A\times \{0\}$ is the last
    coordinate of some tuple in $S_i^\bB$).
    If
    $B'\cap (A\times\{0\}) = \varnothing$, then any element $b\in B'$ is a witness for the existentially quantified variable
    $y$ in $\phi_2$. This proves that if $A\times\{0\}\not\subseteq B'$, then $\bB'$ satisfies $\phi_2$.
    Conversely, if $A\times \{0\} \subseteq B'$, then 
    $\bB'$ does not satisfy $\phi_2$. Therefore, we have to argue that
    $\bA \models \neg \Psi$ if and only if  all substructures $\bB'$ of $\bB$  with
    $A\times \{0\}\subseteq B'$ satisfy $\phi_1$.

    To see this, consider the following mapping $e$ from substructures $\bB'$ of $\bB$ with
    $A\times \{0\}\subseteq B'$ to expansions of $\bA$. For such a $\bB'$, let $\bC := e(\bB')$ be the expansion
    of $\bA$ where $R_i(x_1,\dots, x_{r_i})$ holds if $(x_1,\dots,x_{r_i},i) \not\in B'$,
    equivalently, if $\lnot S_i((x_1,0),\dots, (x_{r_i},0),y)$ holds in $\bB'$ for all $y\in B'$.
   It is straightforward to verify that $e$ is injective. Since $\bB'$ is a substructure of 
   $\bB$, then $\bC$ satisfies the clause $R'(\bar x)\Rightarrow R(\bar x)$ for every $R\in \tau$. 
   Moreover, the image of $e$ is exactly the set of expansions of $\bA$ that 
   satisfy $R'(\bar x)\Rightarrow R(\bar x)$ for every $R\in \tau$. Hence, we regard $e$ as a bijection
   between substructures of $\bB$ containing $A\times \{0\}$ and the expansions of $\bA$ satisfying 
   $R'(\bar x)\Rightarrow R(\bar x)$ for every $R\in \tau$. 

    It follows from the definition of $\phi_1$ and of $\bC$ that for every  $a_{j_1},\dots,a_{j_m} \in A$
    \[
    \bC \models \neg \psi(a_{j_1},\dots,a_{j_m}) \text{ if and only  if } \bB' \models \phi_1'((a_{j_1},0),\dots,(a_{j_m},0)).
    \]
    Then, since $A\times\{0\}\subseteq B'$, it 
    follows that $\bB' \models \phi_1$ if and only if $\bC\models \neg (Q_1x_1\dots Q_mx_m. \psi)$.
    Therefore, if $\bA \models \neg \Psi$ and $\bB'$ is a substructure of $\bB$ with $A\times\{0\}\subseteq B'$, 
    then $\bC \models \neg(Q_1x_1\dots Q_mx_m.\psi)$ and thus $\bB' \models \neg\phi_1$. This implies
    that if $\bA \models \neg \Psi$, then $\bB \in \HER(\phi)$. 
    Conversely, suppose $\bA \models \Psi$, and let $\bC$ be an $\{R_1,\dots,R_n\}$-expansion  of $\bA$
    such that $\bC\models Q_1x_1\dots Q_mx_m.\psi$. In particular, $\bC$ satisfies
    $R'(\bar x)\Rightarrow R(\bar x)$. 
    Hence, there is a substructure $\bB'$ of $\bB$ such that $e(\bB') = \bC$. 
    By the previous arguments we conclude that $\bB'\models \neg \phi_1$, and since
    $A\times\{0\}\subseteq B'$, we see that  $\bB'\models \neg \phi_2$, and so $\bB'\models \neg \phi$.
    Therefore, $\bA\models \neg\Psi$ if and only if $\bB\in \HER(\phi)$, and since
    $\bB$ can be constructed in logarithmic space from $\bA$, we have a
    log-space reduction from $\neg \Psi$ to $\HER(\phi)$.

    We now present a log-space reduction from $\HER(\phi)$ to $\neg \Psi$. Recall
    that there are structures $\bD$ and $\bF$ such that $\bD\models \Psi$ and
    $\bF\models\lnot \Psi$ --- we fix $\bD$ and $\bF$ for this reduction. 
    Let $\bB$ be a finite $\rho$-structure. If $\bB$ does not satisfy
    $\phi_3 \land \phi_4 \land\phi_5 \land \phi_6$, then the reduction returns $\bD$.
    Otherwise, let $X\subseteq B$ be the set of elements
    $x$ for which there is no $\overline {w}$ such that $S_i(\overline{w},x)$ holds for some
    $i\in[n]$. In particular, for every element $b$ in $B\setminus X$, there is some tuple
    $\overline w$ such that $\bB\models S_i(\overline w, b)$ for some $i\in[n]$, and since $\bB$
    satisfies $\phi_3$, it follows that every element of $\bB$ that belongs to some tuple of $E$ belongs to $X$. 
    Also, since $\bB$ satisfies $\phi_5$, the $\{E\}$-reduct of the substructure $\bX$ of $\bB$ with domain
    $X$ is a disjoint union of directed paths and directed cycles. Similarly as above, one can
    note that if there is no directed cycle, then all substructures of $\bB$ satisfy $\phi_2$, and hence
    $\bB$ hereditarily satisfies $\phi$. In this case, the reduction returns the fixed $\tau$-structure $\bF$
    that does not model $\Psi$.
    Otherwise, let $X_1,\dots, X_k$ be the subsets of $X$ inducing a directed cycle in the $\{E\}$-reduct of $\bX$.
    We claim that $k = 1$.
    First note that if $x\in X_i$ for some $i\in [k]$,
    then $\bB\models D(x)$, because $\bB$ satisfies $\phi_4$.  Hence, the union $\bigcup_{i\in[k]}X_i \subseteq D$ is linearly
    ordered by $<$, because $\bB$ satisfies $\phi_6$. Moreover,  $\phi_6$ also guarantees that there is at most one edge
    going backwards, i.e., an edge $(y,x)\in E$ such that $x < y$. Clearly, for every linearly ordered
    directed cycle $X_i$ there is at least one edge going backwards, and so $k = 1$.
    We construct a $\tau$-structure $\bA$ with domain $X_1$ as follows. 
    \begin{itemize}
    \item We start from the substructure of $\bB$ with domain $X_1$. 
    \item For every $i \in \{1,\dots,n\}$, add ($x_1,\dots, x_{r_i})$ to the relation $R'_i$
        whenever there is no $y \in B$ such that $S_i(x_1,\dots, x_{r_i},y)$.
    \item Finally, $\bA$ is the $\tau$-reduct of the resulting structure.
    \end{itemize}
    Clearly, $\bA$ can be computed from $\bB$ in logarithmic space. 
    Now, we argue that $\bA \models \neg \Psi$ if and only if every substructure
    of $\bB$ satisfies $\phi$. Since $\bB\models\phi_3\land\phi_4\land \phi_5 \land \phi_6$ and
    $\phi_3\land\phi_4\land\phi_5 \land \phi_6$ is a universal sentence, 
     we have to argue that $\bA \models \neg \Psi$ if and only if every substructure
     of $\bB$ satisfies $\phi_1\lor\phi_2$.
     Similarly as above, notice that a substructure $\bB'$ of $\bB$ satisfies $\phi_2$ if and only if
     $B' \cap X$ is not $X_1$. 

    Let $\bB_1$ be the substructure of $\bB$ with vertex set $(B\setminus X)\cup X_1$. 
    As explained above, $\bB\in \HER(\phi)$ if and only if all substructures $\bB'_1$ of $\bB_1$ 
    with $X_1\subseteq B_1'$ satisfy $\phi_1$. Similarly as above, we define a mapping $e$ 
    from substructures of $\bB_1$ containing $X_1$ to expansions of $\bA$ satisfying
    $R'(\bar x)\Rightarrow R(\bar x)$. In this case,  $e(\bB_1')$ is the expansion of $\bA$
     where $R_i(x_1,\dots, x_{r_i})$ holds whenever there is no $y \in B_1'$ such that
     $S_i(x_1,\dots, x_{r_i},y)$. Notice that in this case, the mapping $e$ is surjective
     but need not be injective. However, it follows with similar arguments as before
     that $\bB_1' \models \phi_1$ if and only if $e(\bB_1')\models \lnot (Q_1x_1\dots Q_mx_m.\psi)$.
     We thus conclude (via the mapping $e$) that $\bB_1$ contains a structure that does not
     model $\phi_1$ if and only if $\bA$ has an expansion that satisfies $Q_1x_1\dots Q_mx_m.\psi$.
     It now follows that $\bA\models\neg\Psi$ if and only if $\bB\in \HER(\psi)$. 
     Thus, $\HER(\phi)$ reduces in logarithmic space to $\neg\Psi$.
\end{proof}

\begin{theorem}\label{thm:meESO-extESO-coHerFO}
    The following classes have the same computational power up to log-space equivalence, 
    i.e., for every problem $P$ in any of the following classes, there exists a log-space equivalent
    problem in each of the other ones.
    \begin{itemize}
        \item Monadic extensional $\ESO$,
        \item extensional $\ESO$, and
        \item complements of problems in $\HerFO$.
    \end{itemize}
\end{theorem}
\begin{proof}
    This statement is an immediate consequence of the fact that monadic extensional $\ESO$
    is a fragment of extensional ESO, of Lemma~\ref{lem:HerFO-UMSO}, and
    of Theorem~\ref{thm:extESO->HerFO}.
\end{proof}

\section{Monotone extensional SNP}
\label{sect:mon-ext-SNP}
Naturally, \emph{extensional SNP} is the fragment of extensional ESO sentences whose first-order
part is universal. Equivalently, an SNP sentence $\Psi$ with first-order part
$\psi$ in conjunctive normal form is extensional if and only if for every 
existentially quantified symbol $R$ there is a distinct symbol $R'\in R$ such that $\psi$
contains the clause $R'(\overline{x})\Rightarrow R(\overline{x})$, and $R'$ does not appear in any
other clause of $\psi$. 

In this section we show that CSPs of finitely bounded structures are captured by 
extensional SNP (Corollary~\ref{cor:FBstructures})
and conclude with the proof of 
Theorem~\ref{thm:muSNP-fbCSP-sucSNPwe}.
For this purpose, it will be convenient to work with a slight generalization of CSPs: instead of $\CSP(\bB)$ for some relational structure $\bB$, we consider 
$\CSP(\phi)$ for some universal formula $\phi$, or, more generally, for some first-order theory $T$ (see, e.g.,~\cite[Chapter 1]{Book}).

An  instance of $\CSP(\phi)$ (and of $\CSP(T)$) is a 
primitive  positive formula $\psi$, and the task is to decide whether  $\phi\land \psi$
(resp.\ $T\cup\{\psi\}$) is satisfiable. Equivalently, the instance is a structure $\bA$ and the
task is to decide if there is a (possibly infinite) model $\bB$ of $\phi$  such that $\bA\to \bB$. 
Note that this generalizes $\CSP(\bB)$ for finite structure $\bB$, since for every finite structure $\bB$ there exists a universal sentence $\phi$ such that $\CSP(\bB)$ and $\CSP(\phi)$ is the same computational problem; the converse is not true in general.
We first characterize monotone extensional SNP 
in terms of constraint satisfaction problems of universal formulas (Theorem~\ref{thm:CSPs-SUSNP}). 

\subsection{Finitely bounded structures and full homomorphisms}

\label{sect:full}
This section contains a technical lemma that we use to characterize 
the expressive power of monotone extensional $\SNP$ in terms of CSPs of universal sentences
(Theorem~\ref{thm:CSPs-SUSNP}). On the way, we obtain a new (and simpler) proof of a theorem
proved by Ball, Ne\v{s}et\v{r}il, and Pultr~\cite[Theorem 3.3]{ballEJC31}.

Consider a $\tau$-structure $\bA$. We say that two distinct vertices $a,b\in A$ are
\emph{twins} if for every $R\in \tau$ of arity $r$, for $(a_1,\dots, a_r)\in A^r$
and every $i\in[r]$, it is the case that $(a_1,\dots, a_{i-1},a,a_{i+1},\dots, a_r)\in R^\bA$
if and only if $(a_1,\dots, a_{i-1},b,a_{i+1},\dots, a_r)\in R^\bA$. Following graph theoretic
nomenclature, we say that $\bA$ is \emph{point-determining} if $\bA$ contains no twins. 

Notice that for every finite relational signature $\tau$, there is an existential first-order
$\tau$-formula $\NT(x,y)$ such that $\bA \models \NT(a,b)$ if and only if $a$ and $b$ are not twins in $\bA$. Indeed,  one can construct such a formula
by considering the disjunction $\bigvee_{R\in\tau} \psi_R(x,y)$ where for each
$R\in \tau$ of arity $r$, the formula $\psi_R(x,y)$ equals 
    \[
    \exists z_1,\dots z_r \big ( (R(x,z_2,\dots, z_r) \Leftrightarrow \lnot R(y,z_2,\dots, z_r)) \lor \dots \lor
    (R(z_1,\dots, z_{r-1},x) \Leftrightarrow \lnot R(z_1,\dots, z_{r-1},y)) \big ) .
    \]
In particular, $\NT(x,y)$ implies the inequality $x\neq y$, and if $\bA$ is point-determining,
then $\bA\models \NT(a,b)$ if and only if $a\neq b$.

We say that a structure $\bA^\ast$ 
is a \emph{blow-up} of a structure $\bA$ if $\bA$ is a substructure of $\bB$, and every $b\in B\setminus A$ has
a twin $a$ in $A$. We extend this notion to a class of structures $\calC$, and call the class $\calC^\ast$ of
structures $\bB$ that are a blow-up of some $\bA\in \calC$  the \emph{blow-up of $\calC$}.

Recall that a homomorphism $f\colon \bA\to \bB$ is called \emph{full} (sometimes also \emph{strong}) if it also preserves
the relation $\lnot R$ for each $R\in \tau$. Notice that blow-ups and full homomorphisms are naturally
related: there is a surjective full homomorphism $f\colon \bA\to \bB$ if and only if $\bA$ is isomorphic
to a blow-up of $\bB$. 
It is straightforward to observe that for every finite structure $\bB$, there is (up to isomorphism) a unique
point-determining structure $\bB_0$ such that $\bB_0$ embeds into $\bB$ and $\bB$ is isomorphic to a blow-up of
$\bB_0$ --- such a structure can be obtained by 
repeatedly removing twins from $\bB$. 
It follows that if $\calC$ is a hereditary class, then $\bB\in \calC^\ast$ if and
only if $\bB_0 \in \calC$.

\begin{lemma}\label{lem:finite-bound->blow-up}
    Let $\calC$ be a hereditary class of relational structures with a finite domain and signature.
    If $\calC$ is finitely bounded, then the blow-up of $\calC$ is finitely bounded.
    Moreover, there is a universal sentence $\phi$ without equalities such that
    a finite structure $\bA$ satisfies $\phi$ if and only if $\bA\in \calC^\ast$.
\end{lemma}
\begin{proof}
    Given a finite $\tau$-structure $\bA$, the \emph{diagram} of $\bA$ is the $\tau$-formula
    $\diam_{\bA}$ defined as follows. 
    \begin{itemize}
        \item $\diam_{\bA}$ has a free variable $x_a$ for every $a\in A$, 
        \item  for every $R\in \tau \cup\{\neq\}$
    of arity $r$ and each tuple $(a_1,\dots, a_r)$ of $A$, the formula $\diam_{\bA}$ contains the conjunct $R(x_{a_1},\dots, x_{a_r})$
    if $\bA\models R(a_1,\dots, a_r)$,
       \item for every $R\in \tau$
    of arity $r$ and each tuple $(a_1,\dots, a_r)$ of $A$, the formula $\diam_{\bA}$ contains the conjunct $\lnot R(x_{a_1},\dots, x_{a_r})$
    if $\bA\models \lnot R(a_1,\dots, a_r)$. 
    \end{itemize}
    It is straightforward to observe
    that a $\tau$-structure $\bB$ satisfies $\exists \overline x.\diam_{\bA}(\overline x)$ if and only if
    $\bA$ embeds into $\bB$. 
    
    Building on the concept of the diagram of $\bA$, we construct the \emph{point-determining diagram} of
    $\bA$, which is the formula $\PDdiam{\bA}$ obtained from $\diam_{\bA}$
    by replacing each conjunct $x_a\neq x_b$ by the formula $\NT(x_a,x_b)$. We claim that the following
    statements are equivalent.
    \begin{itemize}[itemsep = 0.2pt]
        \item $\bB \models \exists \overline x.\PDdiam{\bA}(\overline x)$,
        \item $\bB_0 \models \exists \overline x.\PDdiam{\bA}(\overline x)$, and
        \item $\bA$ embeds into $\bB_0$. 
    \end{itemize}
    Since $\bB_0$ is a point-determining structure, $\bB_0$ satisfies $\NT(a,b) \Leftrightarrow a\neq b$
    for every pair $a,b\in B_0$. Hence, for every tuple $\overline b$ of $\bB_0$ it is the case that
    $\bB_0 \models \PDdiam{\bA}(\overline b)$ if and only if $\bB \models \diam_\bA(\overline b)$.
    Therefore, the last two itemized statements are equivalent, and since $\bB_0$ embeds into $\bB$,
    it follows that the second itemized statement implies the first one. Now,
    suppose that there is a tuple $\overline b$
    such that $\bB \models \PDdiam{\bA}(\overline{b})$. Let $f\colon \bB\to \bB_0$ be the full homomorphism
    such that $f(a) = f(b)$ if and only $a = b$ or if $a$ and $b$ are twins. It follows from the definition
    of $f$ that $\bB\models \NT(a,b)$ if and only if $\bB_0\models \NT(f(a),f(b))$, and since $f$ is a full homomorphism
    for every $R\in \tau$ of arity $r$ and each $r$-tuple $\overline x$ of $B$ it is the case that
    $\bB\models R(\overline x)$ if and only if $\bB_0\models R(f(\overline x))$. Therefore, since
    $\bB\models \PDdiam{\bA}(\overline b)$, we conclude that $\bB_0\models \PDdiam{\bA}(f(\overline{b}))$. 
    
    To show the statement of the lemma, let $\calF$ be a finite set of bounds of $\calC$, and let $\phi$ be the disjunction
    of the point-determining diagrams $\PDdiam{\bF}$ ranging over all $\bF\in \calF$. It follows from 
    the claim above that a finite structure $\bB$ models $\exists \overline x.\phi(\overline x)$ if and only
    if there is a structure $\bF\in \calF$ such that $\bF$ embeds into $\bB_0$. Equivalently, 
    $\bB\models \exists \overline x.\phi(\overline x)$ if and only if $\bB_0$ does not belong to $\calC$. 
    Since $\bB \in \calC^\ast$ if and only if $\bB_0\in \calC$, we conclude that 
    $\bB \models \forall x.\lnot \phi(\overline x)$ if and only if $\bB \in \calC^\ast$. This proves
    the moreover statement of this lemma. In particular, $\bB\in \calC^\ast$
    if and only if every substructure of $\bB$ with at most $|x|$ vertices belongs to $\calC^\ast$. 
    It thus follows that $\calC^\ast$ is finitely bounded. 
\end{proof}

Before stating our intended application of this lemma, we state two implications
of Lemma~\ref{lem:finite-bound->blow-up} of independent interest.
We denote by $\CSP_{\full}(\bB)$ the class of finite structures that admit a full
homomorphism to $\bB$. \emph{Closure under inverse full
homomorphisms} is defined analogously to closure under homomorphisms.

\begin{corollary}\label{cor:fin-bd}
    Let $\bB$ be a relational structure with finite signature. If $\bB$ is finitely bounded, 
    then $\CSP_{\full}(\bB)$ is finitely bounded.
\end{corollary}
\begin{proof}
    This follows from Lemma~\ref{lem:finite-bound->blow-up},  because $\CSP_{\full}(\bB)$ is the blow-up
    of the age of $\bB$.
\end{proof}

The second application of Lemma~\ref{lem:finite-bound->blow-up} is a new 
proof of the following theorem proved by Ball, Ne\v{s}et\v{r}il, and Pultr~\cite[Theorem 3.3]{ballEJC31}. 
Using their notation, given
a set of structures  $\calB$, we denote the union  $\bigcup_{\bB\in \calB}\CSP_{\full}(\bB)$ by $\CSP_{\full}(\calB)$. 

\begin{corollary}
    For every finite set  $\calB$ of finite relational structures with finite signature, there is a finite
    set of structures $\calA$ such that $\Forb_{\full}(\calA) = \CSP_{\full}(\calB)$. 
\end{corollary}
\begin{proof}
    Let $\calC$ be the union of the ages of $\bB$ over all $\bB\in \calB$. Notice that if $m$ is the size of the
    largest structure in $\calB$, then no structure on at least $m +1$ vertices belongs to $\calC$, and hence $\calC$
    is finitely bounded. Since $\CSP_{\full}(\calB)$ is the blow-up of $\calC$,
    there is a finite set of bounds $\calA$ to $\CSP_{\full}(\calB)$ (Lemma~\ref{lem:finite-bound->blow-up}).

    Moreover, since $\CSP_{\full}(\calB)$ is preserved under inverse full homomorphisms, 
    we may assume that every structure in $\calA$ is point-determining, and thus, a structure $\bA\in \calA$ 
    embeds into a structure $\bC$ if and only if there is a full homomorphism $\bA\to \bC$. The claim now follows. 
\end{proof}

\subsection{Monotone extensional SNP and CSPs}

Given a finite relational signature $\tau$ and a first-order $\tau$-sentence $\phi$, we denote by $\CSP(\phi)$
the class of finite $\tau$-structures $\bA$ for which there is a
$\tau$-structure $\bB$ such that $\bB\models \phi$ and $\bA \to \bB$.
There are first-order sentences such that $\CSP(\phi)$ is undecidable.\footnote{To see this, suppose for contradiction that
 $\CSP(\phi)$ is decidable for every first-order formula $\phi$. Notice that 
the structure $\bA$ with one-element domain and empty interpretation of every $R\in \tau$ belongs to $\CSP(\phi)$
if and only if there is a non-empty model of $\phi$.
Hence, given a first-order formula $\phi$ 
we can check whether $\bA\in\CSP(\phi)$ to decide whether $\phi$ is satisfied in a non-empty structure, contradicting
the undecidablility of first-order logic~\cite{DecisionProblem}.} 
Nonetheless, if one restricts to a universal sentence $\phi$, then $\CSP(\phi)$ is in NP. Indeed, suppose that for
input $\bA$ there is a structure $\bB$ such that $\bA\to \bB$ and $\bB\models \phi$. Since $\phi$ is
universal, it follows that any substructure of $\bB$ also models $\phi$, and thus, by guessing a structure
on at most $|A|$ vertices, one can verify in polynomial time whether $\bA\in \CSP(\phi)$. We propose the following characterization
of monotone extensional SNP.

\begin{theorem}\label{thm:CSPs-SUSNP}
    For a class $\calC$ of relational structures with finite signature, the following statements are equivalent. 
    \begin{itemize}
        \item $\calC$ is in monotone extensional $\SNP$. 
        \item There is a universal sentence $\phi$ such that $\calC = \CSP(\phi)$. 
    \end{itemize}
\end{theorem}
\begin{proof}
    Let $\tau$ be a finite relational signature and  let $\calC$ be a class of finite $\tau$-structures.
    First, suppose that there is a monotone extensional $\SNP$ $\tau$-sentence $\Phi$ describing
    the class $\calC$. Let $\sigma$ be the signature containing exactly the existentially quantified
    symbols from $\Phi$ and let $\tau'\subseteq \tau$ be the signature that contains the symbols $R'$ that are extended by
    some $R\in \sigma$. Let $\nu \land \psi$ be the first-order part of $\Phi$ where $\nu$ is 
    the conjunction of the clauses $R'\Rightarrow R$ for every $R\in \sigma$ and
    $\psi$ is a universal $(\sigma\cup \tau\setminus \tau')$-formula. Let $\psi'$ be the $\tau$-formula
    obtained from $\psi$ by replacing each existentially quantified symbol $R$ by $R'$. We claim that a
    finite $\tau$-structure $\bA$ satisfies $\Phi$ if and only if $\bA\in \CSP(\psi')$. 
    
    Let $\bB$ be an expansion of $\bA$ witnessing that $\bA$ satisfies $\Phi$. 
    Consider the structure $\bA'$ with domain $A$ where $R^{\bA'} = R^\bA$ for every
    $R\in \tau\setminus \tau'$, and where 
    $(R')^{\bA'} = R^{\bB}$
    for every $R'\in \tau'$. 
    Clearly, $\bA'\models \psi'$ because $\bB \models \psi$,
    and since $\bB \models \nu$, it follows that the identity on $A$ defines a homomorphism
    $\bA \to \bA'$, and thus $\bA\in \CSP(\psi')$. 
    
    Conversely, suppose that $\bA\in \CSP(\psi')$, 
    and let $\bB$ be a model of $\psi'$ such that $\bA \to  \bB$. Since $\psi'$ is a universal sentence, 
    we may assume without loss of generality that $\bB$ is a finite structure. Moreover, since $\Phi$
    is in monotone extensional $\SNP$, it is preserved under inverse homomorphisms (see, e.g.,
    Theorem~\ref{thm:SNP-CSP}). Thus, in order to show that $\bA$ satisfies $\Phi$,
    it suffices to show that every finite structure $\bB$ that models $\psi'$ also models $\Phi$. 
    This is immediate: consider the $\sigma\cup\tau$-expansion $\bB'$ of $\bB$ where there interpretation
    of each $R\in \sigma$ coincides with the interpretation of $R'$ in $\bB$; clearly, $\bB'\models \nu$, 
    and since $\bB\models \psi'$, it follows that $\bB'\models \psi$. This proves that the first item implies
    the second one.

    Now, suppose that the second itemized statement holds, and let $\calB$ be the class of finite structures
    that model $\phi$. Again, since $\phi$ is a universal $\tau$-sentence, a $\tau$-structure $\bA$ belongs
    to $\CSP(\phi)$ if and only if there is a structure $\bB\in \calB$ such that $\bA\to \bB$. In particular,  
    every blow-up $\bB^\ast$ of a structure $\bB\in \calB$ belongs to $\CSP(\phi)$ (see Section~\ref{sect:full}). In turn, this implies
    that a structure $\bA$ belongs to $\CSP(\phi)$ if and only if there is an injective
    homomorphism $f\colon \bA\to \bB$ for some  structure $\bB\in \calB^\ast$. 
    Let $\psi'$ be a universal $\tau$-sentence
    without equalities that defines $\calB^\ast$ (see Lemma~\ref{lem:finite-bound->blow-up}). Consider a 
    signature $\sigma$ containing a relation symbol $R^\ast$ for each $R\in \tau$, and the formula $\psi$
    obtained from $\psi'$ by replacing each occurrence of a symbol $R\in \tau$ by $R^\ast$. Finally, 
    let $\Psi$ be the $\SNP$ $\tau$-sentence with existentially quantified symbols
    from $\sigma$ and first-order part $\nu \land \psi$, where $\nu$ is the conjunction of clauses $R \Rightarrow R^\ast$
    ranging over $R\in \tau$. Similarly as above  one can notice that a finite structure $\bA$ satisfies $\Phi$
    if and only if $\bA\in \CSP(\phi)$: every $\bB \in \calB^\ast$ satisfies $\Psi$ (indeed, define an expansion of $\bB$
    satisfying $\nu\land \psi$ by letting $R^\ast = R$ for each $R^\ast\in \sigma$), and since $\Psi$ is in monotone $\SNP$,
    we conclude that every structure $\bA$ in $\CSP(\phi)$ satisfies $\Psi$. Conversely, if $\bB$ satisfies
    $\Phi$, then by considering the structure $\bB'$ where $R^{\bB'} = (R^\ast)^\bB$ for each $R\in \tau$, we obtain
    a structure $\bB'$ that satisfies $\phi$, and such that $\bB\to\bB'$.
    Clearly, $\Phi$ is in monotone extensional $\SNP$, and therefore, the second itemized statement
    implies the first one. 
\end{proof}

Recall from the introduction that a structure $\bB$ is finitely bounded if and only if there exists
a finite set ${\mathcal F}$ of finite structures such that a finite structure $\bA$ embeds into $\bB$ if
and only if no structure from ${\mathcal F}$ embeds into $\bA$. Note that CSPs of finitely bounded structures
are in NP.
 
\begin{corollary}\label{cor:FBstructures}
    Let $\bS$ be a relational structure in a finite signature. If $\bS$ is finitely bounded, 
    then $\CSP(\bS)$ is in monotone extensional $\SNP$. Moreover, if $\bS$ has a finite
    set of connected bounds, then $\CSP(\bS)$ is in monotone connected extensional $\SNP$. 
\end{corollary}
\begin{proof}
    Since $\bS$ is finitely bounded, there is a universal sentence $\phi$ such that a finite structure
    models $\phi$ if and only if $\bA$ embeds into $\bS$. Clearly,  $\CSP(\bS) \subseteq \CSP(\phi)$.  We briefly argue
    that this inclusion is actually an equality. Let $\bB$ be a model of $\phi$ such that there is homomorphism
    $f\colon \bA\to \bB$.  Since $\phi$ is universal, it follows that the substructure $\bB'$ of $\bB$ with vertex
    set $f[A]$ is a finite model of $\phi$ such that $\bA\to \bB'$. 
    We can conclude that $\bA\in \CSP(\bS)$, because
    $\bB'$ is a finite structure satisfying $\phi$, so $\bB'$ embeds into $\bS$, and thus $\bA \to  \bS$.
    The claim of the corollary now follows from Theorem~\ref{thm:CSPs-SUSNP}. 
    Finally, the moreover statement follows analogously by noticing that the construction of the
    formula $\phi$ from Lemma~\ref{lem:finite-bound->blow-up} satisfies that if the finite bounds of 
    $\calC$ are connected, then $\phi$ is a connected universal formula defining $\calC^\ast$. 
\end{proof}

\subsection{CSPs of finitely bounded structures}
Corollary~\ref{cor:FBstructures} asserts that CSPs of finitely bounded structures are in 
monotone extensional SNP. However, there are monotone extensional SNP
sentences that do not describe the CSP of a finitely bounded structure, because the class of finite
models of such a sentence need not be preserved under disjoint unions.
In this section we show that nonetheless, CSPs of finitely bounded structures capture the computational
power of monotone extensional SNP, up to polynomial-time conjunctive-equivalence.

Theorem~\ref{thm:SNP-CSP} guarantees that every monotone connected extensional SNP sentence
$\Phi$ describes the CSP of some structure. Moreover, it follows from the proof of
Theorem~\ref{thm:CSPs-SUSNP} that $\Phi$ describes the CSP of a finitely bounded structure.

\begin{corollary}\label{cor:SUCSNP}
    Every monotone connected extensional SNP sentence $\Phi$ describes a problem of the
    form $\CSP(\bB)$ for some finitely bounded structure $\bB$.
\end{corollary}

In Appendix~\ref{ap:fdCSP->mceSNP}, we include some results regarding the connected fragment
of monotone extensional SNP. In particular, we show that every finite domain CSP is described 
by a monotone connected extensional SNP sentence (Corollary~\ref{cor:finite-domain-MCESNP}),
and we leave as an open question whether for every finitely bounded structure $\bB$, there
is a monotone connected extensional SNP sentence describing $\CSP(\bB)$. We also
show that for every finitely bounded CSP there is a polynomial-time equivalent problem
described by a monotone connected extensional SNP sentence (Lemma~\ref{lem:B->BB2}); we 
use this fact in our next proof.

\begin{theorem}\label{thm:muSNP-fbCSP-sucSNPwe}
     The following classes have the same computational power up to polynomial-time conjunctive-equivalence, i.e.,
    for every problem $P$ in any of the following classes, there are polynomial-time conjunctive-equivalent problems
    in each of the other ones. 
    \begin{itemize}
        \item Monotone extensional $\SNP$.
        \item Monotone connected extensional $\SNP$.
        \item $\CSP$s of finitely bounded structures. 
    \end{itemize}
    Moreover, the last two classes have the same computational power up to log-space equivalence. 
\end{theorem}
\begin{proof}
    It follows form Corollary~\ref{cor:SUCSNP} that every problem in the second itemized class
    belongs to the third one. It follows from Corollary~\ref{cor:FBstructures} that if $\bB$
    is a finitely bounded structure, then $\CSP(\bB)$ is in monotone extensional $\SNP$.
    By Lemma~\ref{lem:B->BB2} we see that if $\CSP(\bB)$ is in monotone extensional SNP,
    then there is a connected monotone extensional SNP sentence $\Phi$ which is log-space equivalent
    to $\CSP(\bB)$. This proves that the ``moreover'' statement holds. 
    Clearly, every problem in the second itemized class belongs to the first itemized class
    of problems. Finally, by Lemma~\ref{lem:B->BB2} we know that for every monotone extensional 
    $\SNP$ sentence $\Phi$ there is a polynomial-time conjunctive-equivalent monotone extensional 
    $\SNP$. 
\end{proof}

\section{Candidate NP-intermediate problems in extensional ESO}
\label{sec:NP-intermediate}
Here we consider the well-known computational problems of
Graph Isomorphism~\cite{groheComACM63} and Monotone Dualization~\cite{eiterDAM156}.
The complexity of each of these problems has been open for more than 50 years, and
some authors regard them as candidate NP-intermediate problems
(see, for instance, the Wikipedia entry for candidate NP-intermediate problems).
We show that for each of these problems there is an extensional ESO sentence
$\Phi$ such that $\Phi$ is polynomial-time equivalent to the corresponding
problem.

\paragraph{Graph Isomorphism.} 
The graph isomorphism problem (GI) is one of the most studied problems
in algorithmic graph theory (see, e.g.,~\cite{groheComACM63} for an overview on this problem).
One of the major breakthroughs in this area is the quasipolynomial-time algorithm
presented by Babai~\cite{babaiGI}. Despite intensive research over the past decades, the
complexity of the graph isomorphism problem remains one of the problems in NP that is not known to be in P,
and thus, it is considered a candidate NP-intermediate problem.

It is well known that the graph isomorphism problem is polynomial-time equivalent
to the graph isomorphism completion problem. The latter  takes as an input
a pair of graphs $(\bG,\bH)$ and a partial function $f'\colon U\subseteq G\to H$,
and the task is to decide if there is an isomorphism $f\colon \bG\to \bH$ such that
$f_{|U} = f'$. We now show that the graph isomorphism completion problem is in extensional ESO.

    Let $U_1$ and $U_2$ be unary relation symbols, $E$ a binary relation symbol,
    and $I$ and $I'$ binary relation symbols that will encode  isomorphisms
    and partial isomorphisms, respectively. Here, the input signature is
    $\tau:=\{U_1,U_2,E,I'\}$.
    Let $\textrm{ISO}(I)$ be a first-order $\{U_1,U_2,E\}$-formula stating
    that $I$ defines an isomorphism from the graph with vertices in $U_1$
    to the graph with vertex set $U_2$, e.g.,
    \begin{align*}
        \textrm{ISO}(I):= \forall x,y\; &(I(x,y)\Rightarrow U_1(x)\land U_2(x)) \\
         \land \; \forall x,y,z\; &\big ((I(x,y)\land I(x,z) \Rightarrow y = z) \land (I(y,z)\land I(y,z) \Rightarrow x = y) \big )\\
        \land \; \forall x \exists y\; & \big ( (U_1(x)\Rightarrow I(x,y))\land (U_2(x)\Rightarrow I(y,x)) \big ) \\
        \land \; \forall x_1,x_2,y_1,y_2\; & \big ( (E(x_1,y_1) \land I(x_1,x_2)\land I(y_1,y_2)) \Rightarrow E(y_1,y_2) \big ) .
    \end{align*}
    Hence, if $\phi$ is a first-order $\{U_1,U_2,E\}$-sentence stating that $E$ is a
    symmetric irreflexive relation that connects only vertices within $U_1$ or within $U_2$, 
    we get that Graph Isomorphism Completion can be expressed by the extensional ESO sentence
    \[
        \Psi:=\exists I.\; \forall x,y \big ((I'(x,y)\Rightarrow I(x,y)) \land \textrm{Iso}(I)\land \phi \big ).
    \]
    Clearly, every GI completion instance $X$ can be encoded as a $\tau$-structure $\bA$
    such that $\bA\models \Psi$ if and only if $X$ is a yes-instance. Also notice that
    if an input $\tau$-structure $\bA$ is such that $I'$ does
    not define a partial isomorphism from the graph with vertex set $U_1(\bA)$ to the
    graph with vertex set $U_2(\bA)$, then $\bA\not\models \Psi$. In this way, we obtain
    a polynomial-time equivalence between the graph isomorphism completion problem and deciding
    $\Psi$.

\paragraph{Monotone Dualization.} We assume basic familiarity with propositional logic
(see, e.g.,~\cite[Section~2]{eiterSICOMP32}).
Our presentation of the monotone dualization problem is similar to the presentation given in~\cite{eiterDAM156},
and we refer to that survey for a detailed exposition.
The \emph{dual} of a boolean function $f\colon\{0,1\}^n\to \{0,1\}$
is the function $f^d(x_1,\dots, x_n):= \bar f(\bar x_1, \dots, \bar x_n)$.
In particular, if $\phi$ is a CNF formula representing $f$ (where the constant $0$
is represented by an empty clause, and the constant $1$ is represented by an empty
set of clauses),  then the dual $f^d$ of $f$ can be easily obtained by interchanging
the symbols $\land$ and  $\lor$. 
Computing the CNF of the dual of a boolean function $f$ presented in CNF traces back
to the middle of the last century (see, e.g,~\cite{maghout1959}), and it is one of
few problems whose complexity remains unclassified~\cite{eiterDAM156}.

A boolean function $f\colon\{0,1\}^n\to \{0,1\}$ is \emph{monotone} if 
$f(\bar x)\le f(\bar y)$ whenever $x_i\leq y_i$ for each argument $i\in[n]$. 
Equivalently, a boolean function is monotone if it can be represented by
a boolean formula $\phi$ in CNF with no negative literals.

\vspace{0.4cm}
\noindent\textbf{Monotone Dual}

 \textsc{Input:} Functions $f$ and $g$ represented by
 formulas $\phi$ and $\psi$ in CNF without negative literals.
 
 \textsc{Question:} If $g$ the dual of $f$?
\vspace{0.3cm}

Clearly, this problem is in coNP: if there is an evaluation
$v_1,\dots, v_x$ of the variables $x_1,\dots, x_n$ such that
$\phi^d(v_1,\dots, v_n)$ is true, and $\psi(v_1,\dots, v_n)$ is
false, then $g$ is not the dual of $f$ (here, $\phi^d$ is 
the boolean formula in DNF obtained by swapping the
symbols $\land$ and $\lor$ in $\phi$). This
problem can be solved in quasi-polynomial time~\cite{fredmanJA21}, and
according to~\cite{eiterDAM156} is has been open for over 40
years now, whether it can be solved in polynomial time.
We construct an extensional ESO sentence $\Phi$
such that deciding $\Phi$ is polynomial-time equivalent
to the complement of Monotone Dual, i.e., deciding whether 
$g$ and $f$ are not duals.

To begin with, we will consider a boolean formula
$\phi$ in CNF or DNF, and without 
negative literals, empty clauses, and empty monomials.
We will encode such formulas as
structures $F(\phi)$
with the binary signature $\{C,EQ\}$ as follows.
The vertices are occurrences of variables (i.e., for each
clause $c$ and each variable $x$ in $c$, there is a vertex
$v_{x,c}$), and the interpretations of $C$ and $EQ$ are
equivalence relations where 
\begin{itemize}
\item each class of $C$  consists of all occurrences of a variable appearing
in the same clause, and 
\item each class of $EQ$ consists of the occurrences of the same variable.
\end{itemize} 
In symbols, $(v_{x,c},v_{y,c'})\in C$ iff $c = c'$, and
$(v_{x,c},v_{y,c'})\in EQ$ iff  $x = y$.
In particular,  if $\phi$ is in CNF and $\phi^d$ is the dual DNF formula, 
then $F(\phi) = F(\phi^d)$. Also, for every $\{C,EQ\}$-structure
$\bA$ where $C$ and $EQ$ are equivalence relations, there is
a monotone CNF formula $\phi$ such that $F(\phi) = \bA = F(\phi^d)$.

We now consider a unary predicate $U_1$ 
that will encode an evaluation of the variables in the
formula $\phi$. In order for $U_1$ to encode 
a valid evaluation we must impose the first-order restriction
stating that if $(x,y)\in EQ$, then $x\in U_1$ if and only
if $y \in U_1$ (i.e., different occurrence of the same variable
must receive the same value). So let $\delta$ be
the (universal) first-order $\{U_1,C,EQ\}$-sentence stating
that $C$ and $EQ$ are equivalence relations and that $U_1$ is
preserved by the relation $EQ$ in the sense explained above. Clearly, structures
that satisfy $\delta$ are in one-to-one correspondence with
evaluations of formulas in CNF (and in DNF) without negative literals, empty clauses, and empty monomials.

Consider the following first-order $\{U_1,C,EQ\}$-sentences
\begin{align*}
    \mathtt{true_C}:= \forall x\exists y\; (C(x,y) \land U_1(y))~~ \text{ and }~~\mathtt{true_D}:=\exists x\forall y\; (C(x,y) \Rightarrow U_1(y)).
\end{align*}
Observe that if $\phi$ is a monotone CNF formula with
variables $X$, then an evaluation $e\colon X\to \{0,1\}$ 
makes $\phi$ true if and only if the structure
$(F(\phi), U_1:=e^{-1}(1))$ satisfies $\mathtt{true_C}$. 
Similarly, $e$ 
makes the dual formula $\phi^d$ in DNF true
if and only if $(F(\phi), U_1=e^{-1}(1))\models \mathtt{true_D}$. 
In fact, the 
$\{U_1,C,EQ\}$-expansions of $F(\phi)$ that satisfy
$\mathtt{true_C}$ (resp.\  $\mathtt{true_D}$)
are in one-to-one correspondence with the
satisfying evaluations of $\phi$ (resp.\ of $\phi^d$).

To obtain an extensional ESO sentence that is polynomial-time
equivalent to the problem Monotone Dual, we need three more unary
symbols $U_f, U_g,$ and $U_1'$. Let $\rho$ be a first-order
sentence stating the following:
\begin{itemize}
    \item the  interpretations of $U_f$ and $U_g$
    are total and disjoint, i.e., every vertex belongs to 
    exactly one of $U_f$ and $U_g$, 
    \item if a pair of vertices $x$ and $y$ are connected by an 
    edge in $C$, then $x$ and $y$ both belong to $U_f$ or
    both to $U_g$, and
    \item for each vertex $x\in U_f$ there is a vertex
    $y\in U_g$ such that $xy\in EQ$, and vice versa.
\end{itemize}
Given a pair of CNF formulas $\phi$ and $\psi$ on the 
same set of variables without negative literals and empty clauses, we construct a
structure $F(\phi,\psi)$ with  signature $\{U_f, U_g, U_1', C, EQ\}$
as follows.
This structure is obtained from the disjoint union  $F(\phi)+ F(\psi)$
where  we colour the vertices in  $F(\phi)$ with $U_f$ and  the vertices in 
$F(\psi)$ with $U_g$, the interpretation of $U_1'$ is empty, and we 
modify the interpretation of $EQ$ by connecting all pairs of vertices $x$
and $y$ that represent the same variable 
(in particular, $EQ(F(\phi))\cup EQ(F(\psi))\subseteq EQ(F(\phi,\psi))$).
Clearly, $F(\phi,\psi)\models\rho\land \delta$.

Finally, for $\mathtt{X}\in \{\mathtt{C,D}\}$ and $h\in \{f,g\}$,
we denote by $\mathtt{true_X}(U_h)$ the restriction of $\mathtt{true_X}$
to the substructure with domain $U_h$ (which can be written as a
first-order sentence with symbols in $\{C,EQ, U_1, U_h\}$). With these
sentences we define the extensional ESO sentence
\begin{align*}
     \Gamma:= \exists U_1  \big ( (U'_1\Rightarrow U_1) \land & \; \rho \land\delta \\
    \land & \left( (\mathtt{true_D}(U_f)\land \lnot \mathtt{true_C}(U_g))
    \lor (\lnot \mathtt{true_D}(U_f)\land  \mathtt{true_C}(U_g)) \right ) \big).
\end{align*}
A structure of the form $F(\phi, \psi)$ 
satisfies $\Gamma$ if and only if there is an evaluation of the variables that either make
exactly one of $\phi^d$ or $\psi$ true, and the other one false.
Therefore, on an input $(\phi,\psi)$ to Monotone Dual, where $\phi$ and $\psi$ are
formulas CNF formulas without negative literals and empty clauses that represent
the Boolean functions $f$ and $g$, respectively,
we have that $g$ is not the dual of $f$ if and only if $F(\phi, \psi)$ satisfies
$\Gamma$. This gives us a polynomial-time reduction from
Monotone Dual to deciding $\lnot \Gamma$.

Now we show that conversely,
deciding $\lnot \Gamma$ reduces in polynomial time to Monotone Dual.
First, fix a pair of formulas $\phi_X$ and
$\psi_X$ in monotone CNF over the same set of variable such that
$\psi_X$ is not the dual of $\phi_X$. The reduction works as follows. 
Given a $\{U_1', U_f, U_g, C, EQ\}$-structure $\bA$ we first
first verify whether $\bA\models \rho\land\delta$, and if not, the reduction
returns the pair $(\phi_X, \psi_X)$.  Otherwise, we construct $\phi$
and $\psi$ as follows. 
The variables of $\phi$ and of $\psi$ are the $EQ$-equivalence
classes, and  for each $C$-equivalence class $Y$ in $U_f(\bA)$
(resp.\ in $U_g(\bA)$) we add a clause in $\phi$ (resp.\ in $\psi$)
containing the $EQ$-equivalence classes that intersect $Y$. Finally, 
if some vertex $v$ is colored with $U_1'$, then we remove all clauses from $\phi$ and $\psi$ that contain $v$.
It readily follows from the arguments above that $\bA\models\Gamma$ if
and only if $\phi$ is not the dual of $\psi$. Hence, the
reduction that maps $\bA$ to $(\phi_X,\psi_X)$ if $\bA\not\models\delta\land \rho$, 
and to $(\phi,\psi)$ otherwise, is a polynomial-time reduction from deciding $\lnot\Gamma$
to the Monotone Dual problem. Both reductions together show that
Monotone Dual is polynomial-time equivalent to deciding $\lnot \Gamma$.

\paragraph{Conjecture.} We have seen that surjective finite-domain CSPs (Section~\ref{ex:finite-CSPs})
and CSPs of finitely bounded structures (Corollary~\ref{cor:FBstructures}) are expressible in ESO.
Moreover,  we proved that extensional ESO has the same computational power as HerFO (Theorem~\ref{thm:meESO-extESO-coHerFO}), 
and so, the tractability meta-problem for extensional ESO is undecidable (Corollary~\ref{cor:tractability-problem}). 
Finally, we argued in this section that extensional ESO can express problems
that are polynomial-time equivalent to Graph Isomorphism, and to
the complement of Monotone Dualization. These observations make us believe
that the following is true.

\begin{conjecture}\label{conj:intermediate}
    There is an extensional ESO sentence $\Phi$ such that
    deciding $\Phi$ is NP-intermediate.
\end{conjecture}

\color{black}

\section{Limitations of the computational power}
\label{sect:limitations}

In this section we show that 
extensional ESO 
is not  NP-rich (unless E = NE), i.e., that there are sets $L$ in NP which 
are not polynomial-time equivalent to a problem in extensional ESO. To do so, we study
disjunctive self-reducibility; a notion in computational complexity that traces back to~\cite{meyerMIT79,ICALP-1976-Schnorr}.

\subsection{Disjunctive self-reducible sets}
Intuitively speaking, a set $L$ is \emph{self-reducible} if deciding whether $x$ belongs to $L$ reduces
in polynomial-time to deciding $L$ for  ``smaller'' instances than $x$. 
Moreover, it is \emph{disjunctive
self-reducible} if  there is a deterministic polynomial-
smaller instances $x_1,\dots, x_n$ such that $x\in L$ if and only if $x_i\in L$ for some $i\in [n]$.
A typical example of a disjunctive self-reducible set is SAT: a boolean formula 
$\phi(x_1,\dots, x_n)$ is satisfiable if and only if at 
$\phi(0,x_2,\dots, x_n)$ and $\phi(1,x_2,\dots, x_n)$ is satisfiable. This concept has played a
major role in the complexity theory literature on the class NP, see, e.g.,~\cite{Ambos-SpiesK88,Hemaspaandra2020}.

Let us formally define disjunctive self-reducible sets (also see also~\cite{meyerMIT79,ICALP-1976-Schnorr}).
A polynomial-time computable partial order $<$ on $\{0,1\}^\ast$ is \emph{OK}
if and only if there is a polynomial $p$ such that
\begin{itemize}
    \item every strictly decreasing chain $x_1 < \cdots < x_n$ of words from $\{0,1\}^\ast$ is shorter than $p(x_n)$, i.e.,
    $n \le p(x_n)$.
    \item for all $x,y\in \{0,1\}^\ast$, $x < y$ implies $|x| \le p(|y|)$.
\end{itemize}
A set $L\subseteq \{0,1\}^\ast$ is \emph{self-reducible} if there are an OK partial order
$<$ on $\{0,1\}^\ast$ and a deterministic polynomial-time oracle machine $M$ that accepts $L$ 
and asks its oracle queries about words smaller than the input $x$, i.e.,
on input $x$, if $M$ asks a query to its oracle for $L$ about $y$, then $y < x$. 
We say that $L$ is \emph{disjunctive self-reducible}  
if on every input $x$ the machine $M$ either
\begin{itemize}
    \item decides whether $x\in L$ without calling the oracle, or
    \item computes queries $x_1,\dots, x_n \in \{0,1\}^*$ strictly smaller than $x$ (with respect to $<$)
    such that
    \[
    x\in L \text{ if and only if } \{x_1,\dots, x_n\}\cap L\neq \varnothing.
    \]
\end{itemize}
Clearly, every set in $\PO$ is disjunctive self-reducible, and it follows from~\cite{koJCSS26}
that every disjunctive self-reducible set is in $\NP$.

\begin{lemma}\label{lem:extESO->d-self-reducible}
    If a class of structures $\calC$ is  expressible in extensional $\ESO$, then $\calC$
    is disjunctive self-reducible. In particular, $\CSPs$ of finitely bounded structures are
    disjunctive-self-reducible.
\end{lemma}
\begin{proof}
    We prove the claim for the case that $\Phi$ has one existentially quantified relation symbol 
    $R$ or arity $r$ --- the general case follows with similar arguments. Let $R'$ be the
    input symbol such that $\Phi$ contains the clause $R'(x) \Rightarrow R(x)$ and $R'$ does not
    appear in any other clause of $\Phi$. Given an input $\tau$-structure $\bA$ we first test whether 
    the expansion $\bA'$ with $R^{\bA'} = R'$ satisfies the first-order part of $\Phi$.
    If yes, we accept, and otherwise we return the structures obtained from $\bA$
    by adding one new tuple to the interpretation of $R'$. It is clear that the
    order  $\bB < \bA$ if $A = B$, 
    $S^\bA = S^\bB$ for every $S\in\tau\setminus \{R'\}$, 
    and $(R')^{\bA'}\subseteq  (R')^{\bB'} $  is an OK poset witnessing that the
    previous reduction is a disjunctive self-reduction. 
\end{proof}

In the rest of this section we study disjunctive self-reducible sets. In particular, we
show that the class of disjunctive self-reducible sets is not NP-rich, unless E = NE (Theorem~\ref{thm:d-self-reducible-not-NP-rich});
and also prove that there are NP-intermediate disjunctive self-reducible sets unless P = NP
(Theorem~\ref{thm:non-dicho}). The former, together with Lemma~\ref{lem:extESO->d-self-reducible}
implies that extensional ESO does not have the full computational power of NP
(Corollary~\ref{cor:notNPrich}), unless E = NE.

\subsection{NP-intermediate disjunctive self-reducible sets}

We consider the following family of disjunctive self-reducible sets.

\begin{lemma}\label{lem:uNP-Ffree}
    Let $\tau$ be a finite relational signature and $\calF$ a (not necessarily finite) class of finite
    $\tau$-structures. If $\calF$ is in $\PO$, 
    then the class of structures $\bA$ that contain some substructure $\bF$ from $\calF$ is disjunctive self-reducible.
\end{lemma}
\begin{proof}
    Given an input structure $\bA$, we first test whether $\bA\in \calF$. If yes, we accept, otherwise,
    we return the substructures obtained from $\bA$ by removing exactly one element from the domain. 
    Here, the order $\bA < \bB$ defined by $\bA$ being a substructure of $\bB$ is clearly an OK
    poset for which the previous reduction is a disjunctive self-reduction. 
\end{proof}

For our next result we will use a construction of coNP-intermediate CSPs from
Bodirsky and Grohe~\cite{BodirskyGrohe} which uses the following set of tournaments. 
The \textit{Henson set} is the set ${\mathcal T}$ of tournaments $\mathbb T_n$ defined for
 positive integer $n\ge 5$ as follows. The vertex set of $\mathbb T_n$ is $[n]$
 and it contains the edges
 \vspace{-2.4pt}
 \begin{itemize}[itemsep = 0.8pt]
     \item $(1,n)$,
     \item $(i,i+1)$ for $i\in [n-1]$, and
     \item $(j,i)$ for $j > i+1$ and $(j,i)\neq (n,1)$.
 \end{itemize}

It is straightforward to observe that, for any fixed (possibly infinite) set of tournaments $\calF$
(such as $\calT$), the class of oriented graphs that  do not embed any tournament from $\calF$ 
is of the form $\CSP(\bB_{\mathcal F})$ for some countably infinite structure $\bB_{\mathcal F}$
(we may even choose $\bB_{\mathcal F}$ to be homogeneous). In the concrete situation of the class ${\mathcal T}$ defined above,
$\CSP(\bB_{\mathcal T})$ is an example of $\coNP$-complete CSP (see, e.g.,~\cite[Proposition 13.3.1]{Book}).

\begin{theorem}\label{thm:non-dicho}
There is no $\P$- versus $\NP$-complete dichotomy for disjunctive self-reducible sets.
\end{theorem}
\begin{proof}
    We use  that there exists a subset ${\mathcal T}_0$ of the Henson set ${\mathcal T}$
    such that there exists a linear-time Turing machine $G$ computing a function $g \colon {\mathbb N} \to {\mathbb N}$
    (where the input to $G$ is represented in unary) with 
    \begin{itemize}
        \item ${\mathcal T}_0 = \{\bT_n \mid g(n) \text{ is even} \}$, 
        \item $\CSP(\bB_{{\mathcal T}_0})$ is not in P  and not coNP-hard, unless P = coNP~\cite{BodirskyGrohe}. 
    \end{itemize}
Clearly, $\calT_0$ is in $\PO$: given an input digraph $D$ verify whether $g(|D|)$ is even
(if not, reject), and if this is the case, verify whether $D$ belongs to the Henson set $\calT$ --- which is even
in first-order logic~\cite{ManuelSantiagoCSL}. 
Since $\CSP(\bB_{\calT_0})$ is the class of $\calT_0$-free
loopless oriented graphs, it follows from Lemma~\ref{lem:uNP-Ffree}
that $\CSP(\bB_{\calT_0})$ is disjunctive self-reducible. 
Note that this implies the unconditional non-dichotomy result stated in the theorem:
if P = NP, then there is no dichotomy, and if P is different from NP,
then  there are NP-intermediate disjunctive self-reducible sets.
\end{proof}

\subsection{Search reduces to decision}
\label{sub:search}

For a polynomial-time computable $R \subseteq (\{0,1\}^*)^2$, the \emph{search problem for $R$}
is the computational problem to compute for a given $u \in \{0,1\}^*$  some $v \in \{0,1\}^*$ such
that $(u,v) \in R$;  if there is no such $v$, the output can be arbitrary (even undefined).
Clearly, the search problem for $R$ is at least as hard deciding the language
$$ L_R := \{u\in\{0,1\}^\ast \mid \text{there is } v\in \{0,1\}^\ast \text{ such that } (u,v)\in R\}.$$ 
For a given $R$, we say that \emph{search for $R$ reduces to decision} if 
the search problem for $R$ can be solved in polynomial time given an oracle for the decision of $L$.  
For a given language $L$ in NP, we say that \emph{search for $L$ reduces to decision} if there
exists a polynomial-time computable $R \subseteq (\{0,1\}^*)^2$ such that $L = L_R$
and search for $R$ reduces to decision. 

In particular, for every disjunctive self-reducible set $L$ search for $L$ reduces to decision.
Indeed, suppose that $M$ is a Turing machine witnessing that $L$ is disjunctive self-reducible and $<$ is the
corresponding OK poset. For each $u\in L$ consider witnesses of the form $v_1\dots v_k$ where $u = v_1$, 
$v_i < v_{i+1}$ for each $i\in [k-1]$, $v_k$ is accepted by $M$ without calling the oracle, and $v_{i+1}$
is one of the queries that $M$  computes on input $v_i$ before calling the oracle. Clearly, the set
$R$ of such tuples $(u,v_1\dots v_k)$ is decidable in polynomial-time, and search reduces to decision for
$R$.

\begin{lemma}\label{lem:std-preservation}
    Let $L$ and $L'$ be sets such that search reduces to decision for $L$. If there is a 
    polynomial-time disjunctive reduction from $L'$ to $L$, and a polynomial-time Turing reduction
    from $L$ to $L'$, then search reduces to decision for $L'$. 
\end{lemma}
\begin{proof}
    Let $R$ be a binary relation such that $L = L_R$ and search reduces to decision for $R$, and
    let $f$ be a polynomial-time computable function proving that there is a disjunctive polynomial-time
    reduction from  $L'$ to $L$, i.e., $x\in L'$ if and only if $f(x)\cap L\neq \varnothing$.
    Note that a witness for $x\in L'$ is a pair $(u,v)$ such that $u\in f(x)$ and $(u,v)\in R$.
    Since there is a polynomial-time Turing reduction from $L = L_R$ to $L'$, and search
    reduces to decision for $R$, there is a polynomial-time deterministic oracle machine $M$
    that solves the search problem for $R$ using $L'$ as an oracle. Let $M'$ be the deterministic
    polynomial-time oracle machine that on input $x$ does the following:
    \begin{itemize} 
    \item it first computes $f(x)$; 
    \item for each $u\in f(x)$
    it simulates $M$ to find a witness $v$ such that $(u,v)\in R$; 
    \item if it finds such a $v$ for some $u\in f(x)$,
    then $M'$ outputs $(u,v)$; 
    \item otherwise it rejects. 
    \end{itemize}
    It clearly follows that $M'$ is an oracle machine proving
    that search reduces to decision for $L'$.
\end{proof}

It is now straightforward to observe that search reduces to decision for every problem in any of the
classes considered so far.  

\begin{theorem}\label{thm:search-to-decision}
Search  reduces to decision for every problem $P$ in any of the following classes.
\begin{itemize}
    \item $\CSP$s of finitely bounded structures. 
    \item Monotone extensional $\SNP$.
    \item Extensional $\ESO$.
\end{itemize}
In particular, neither of these classes in disjunctive $\NP$-rich unless 
$\EE = \NEE$.
\end{theorem}
\begin{proof}
    We already argued that search reduces to decision for every disjunctive self-reducible
    sets, and it follows from Lemmas~\ref{lem:extESO->d-self-reducible} and~\ref{lem:uNP-Ffree}
    that any problem in these classes is disjunctive self-reducible. To see that  the last claim of
    this theorem holds,  we use Lemma~\ref{lem:std-preservation} to note that if $L$
    is polynomial-time disjunctive equivalent to $L'$ where $L'$ belongs to one of the itemized classes, 
    then search reduces to decision for $L$. The claim now follows, because
    Ballare and Goldwasser~\cite{BellareGoldwasser} proved that there exists a language $L \subseteq \{0,1\}^*$
    in NP such that search does not reduce to decision for $L$ unless $\EE=\NEE$.
\end{proof}

\subsection{Promise problems}
\label{sub:promise}
Several $\NP$-complete problems become tractable given a strong enough promise.
For instance, $k$-colourability is decidable in polynomial time, given the
promise that the input graph $G$ is \emph{perfect}, i.e., the chromatic number
of $H$ equals the largest complete subgraph of $H$ for every induced subgraph $H$ of
$G$. Hence, to decide whether $G$ is $k$-colourable it suffices to verify whether $G$ is $K_{k+1}$-free,
which is clearly in P. In this example, the promise is even verifiable in polynomial-time~\cite{chusdnovskyAM164},
but this need not be the case in general. A prominent example being the problem 1-in-3 SAT, which 
is tractable given the promise that the input
formula $\phi$ either has a 1-in-3 solution, or does not even have a not-all-equal
solution~\cite{BG21}. In contrast, some problems remain hard even given some non-trivial promise.
A typical example is $3$-colourability which is $\NP$-hard given the promise that the
input graph $G$ is either $3$-colourable or not even $5$-colourable~\cite{bartoJACM68}. 

A general framework for studying hard promises was introduced in~\cite{selmanIC78}.
Formally, a (decidable) \emph{promise problem} is a pair of decidable sets $(Q,R)$ where 
$Q$ is the \emph{promise}, and $R$ is the \emph{property} that needs to be decided
given the promise $Q$. A deterministic Turing machine $M$ \emph{solves} the promise
problem $(Q,R)$ if for every input $x$ the following holds
\[
\text{if } x\in Q, \text{ then } M \text{ accepts } x \text{ if and only if } x\in R.
\]
In this case, we say that the language $L$ accepted by $M$ is a \emph{solution} to $(Q,R)$. 
In particular, notice that $R$ and $Q \cap R$ are solutions to $(Q,R)$. Going back
to our previous examples, if $Q$ is the class of perfect graphs, and $R$ the class
of $k$-colourable graphs, then the class of $K_{k+1}$-free graphs is a polynomial-time
solution to $(Q,R)$.

A promise problem $(Q,R)$ is in $\PO$ if $(Q,R)$ has a solution  $L$ in $\PO$, and
$(Q,R)$ is \emph{NP-hard} if every solution to $(Q,R)$ is NP-hard. Note that for every promise
problem $(Q,R)$, we have that $R$ is a solution to $(Q,R)$, and hence if $R$ is in $\PO$, then $(Q,R)$ is in $\PO$.
A promise problem $(Q,R)$ is \emph{polynomial-time Turing-reducible%
\footnote{Many-one, disjunctive, and conjunctive reductions are defined analogously,
but are not needed for this paper.}} to a promise problem $(S,T)$
if for every solution $A$ of $(S,T)$ there is a solution $B$ of $(Q,T)$ such that $B$
is polynomial-time Turing-reducible to $A$. Notice that if $L$ is a decidable set,
then the only solution to $(\{0,1\}^\ast, L)$ is $L$, and hence, we say that a set
$L$ is polynomial-time Turing-reducible to a promise problem $(Q,R)$ if $(\{0,1\}^\ast, L)$
is polynomial-time Turing-reducible to $(Q,R)$. 

\emph{Natural promise problems} trace back to at least 1984, and they arise in the context
of public-key cryptography~\cite{evenIC61}. Given a decidable set $L \subseteq \Sigma^*$,
for some finite alphabet $\Sigma$, the \emph{natural promise problem} for $L$ is the promise
problem 
\[\left ( L\times (\Sigma^* \setminus L)~\uplus~(\Sigma^* \setminus L) \times L,~L\times \Sigma^\ast \right).\]
That is, an instance to the problem is a pair $(x,y)$, the promise is that exactly one of $x$ or $y$ belongs to $L$,
and the property to decide is whether $x\in L$.
Clearly, the natural promise problem for $L$ is at most as hard as $L$, and thus, if $L$ is in $\PO$
then $L$ and the natural promise problem for $L$ are trivially polynomial-time equivalent.
In~\cite{evenIC61}, the authors  prove that the natural promise problem for SAT is $\NP$-hard, 
and it follows from~\cite{selmanIC78} that $L$ and its natural promise problem are polynomial-time
Turing-equivalent for every $\NP$-complete set $L$. In particular, it follows from the finite domain dichotomy
conjecture~\cite{BulatovFVConjecture,ZhukFVConjecture} that every finite domain CSP is polynomial-time
Turing-equivalent to its natural promise problem -- this will also follow from our next theorem without
using the finite domain dichotomy. The graph isomorphism problem (which
is not known to belong to P and most likely not NP-complete) is polynomial-time equivalent to its natural promise problem~\cite{selmanIC78}.

\begin{theorem}\label{thm:promises}
    For every $L$ in any of the following classes, the natural promise problem
    for $L$ is polynomial-time Turing equivalent to $L$.
    \begin{itemize}
        \item $\CSP$s of finitely bounded structures. 
        \item Monotone extensional $\SNP$.
        \item Extensional $\ESO$.
    \end{itemize}
\end{theorem}
\begin{proof}
    Clearly, every set $L$ is at least as hard as its natural promise problem. 
    On the other hand, Theorem 3 in~\cite{selmanIC78} asserts that if $L$ is disjunctive self-reducible,
    then there is a polynomial-time Turing-reduction from $L$ to the natural promise problem $L$.
    By Lemmas~\ref{lem:extESO->d-self-reducible} and~\ref{lem:uNP-Ffree} we know that
    any problem in $L$ in any of the itemized classes is disjunctive self-reducible, and so,
    the natural promise problem for $L$ is polynomial-time Turing equivalent to $L$.
\end{proof}

\subsection{Failure of NP-richness}

It follows from results in~\cite{selmanIC78} that there are problems in $\NP$ which are
not polynomial-time disjunctive-equivalent to any disjunctive self-reducible set, unless E = NE.%

\begin{theorem}
[essentially in~\cite{selmanIC78}]
\label{thm:d-self-reducible-not-NP-rich}
    The class of disjunctive self-reducible sets is not disjunctive $\NP$-rich, unless $\E = \NE$. 
\end{theorem}
\begin{proof}
    Theorem 3 in~\cite{selmanIC78} asserts that if $L$ is disjunctive self-reducible, then
    the natural promise problem for $L$ is polynomial-time Turing-equivalent to $L$, and Theorem 4
    in~\cite{selmanIC78} asserts that the property ``$L$ and the natural promise problem
    for $L$ are polynomial-time Turing equivalent'' is preserved under polynomial time
    disjunctive-equivalence. Finally, Theorem 7 in~\cite{selmanIC78} guarantees that if $\E \neq \NE$,
    then there is a set $L\in \NP\setminus \PO$ such that the natural promise problem for $L$ is in $\PO$,
    i.e., $L$ and its natural promise problem are not polynomial-time Turing-equivalent.
\end{proof}

\begin{corollary}\label{cor:notNPrich}
    None of the following classes is $\NP$-rich, 
    unless $\E = \NE$. 
    \begin{itemize}
        \item $\CSP$s of finitely bounded structures. 
        \item Monotone extensional $\SNP$.
        \item Extensional $\ESO$.
    \end{itemize}
\end{corollary}

\section{Conclusion and Open Problems}
\label{sect:open}

We  proved that the complement of every problem in extensional ESO is polynomial-time equivalent
to a problem in $\HerFO$, and vice versa (Theorem~\ref{thm:meESO-extESO-coHerFO}). We also 
showed that there are problems in $\coNP$ that are not polynomial-time equivalent to problems in
Extensional ESO, unless NE=E (Corollary~\ref{cor:notNPrich}). 
Our results also show that  there are subclasses of NP that are believed to have
no P versus NP-complete dichotomy, while the previous technique of showing this
by establishing that the class is NP-rich does not work. For instance for
proving that CSPs of finitely bounded structures do not have a complexity dichotomy,
we therefore need a different approach.

Besides Conjecture~\ref{conj:intermediate}, 
a number of open problems are left for future research. 
\begin{enumerate}
    \item Is every disjunctive self-reducible set polynomial-time (disjunctive-) equivalent to an  extensional ESO sentence (see Figure~\ref{fig:landscape})?
    \item Is it true that for every extensional ESO sentence $\Phi$ there is a finitely bounded structure $\bA$ such that $\Phi$ and $\CSP(\bA)$ are
    polynomial-time equivalent (see again Figure~\ref{fig:landscape})? 
    \item Are there $\NP$-intermediate CSPs of finitely bounded structures (assuming $\PO\neq \NP$)?
    \item Is there a finitely bounded structure $\bA$ such that $\CSP(\bA)$ is polynomial-time equivalent 
    to the graph isomorphism problem? 
    \item Is every CSP of a finitely bounded structure in connected monotone extensional SNP? (Compare with
    Corollary~\ref{cor:SUCSNP} and the moreover statement in Theorem~\ref{thm:muSNP-fbCSP-sucSNPwe}).
\end{enumerate}

\section*{Acknowledgments}
The authors are thankful to Antoine Amarilli for a guided tour through
Constraint Topological Sorting.


\bibliographystyle{abbrv}
\bibliography{global.bib}

\begin{thebibliography}{10}

\bibitem{alvaradoAOR280}
J.~Alvarado, S.~Dantas, and D.~Rautenbach.
\newblock Sandwiches missing two ingredients of order four.
\newblock {\em Annals of Operations Research}, 280:47--63, 2019.

\bibitem{amarilliICALP2018}
A.~Amarilli and C.~Paperman.
\newblock {Topological Sorting with Regular Constraints}.
\newblock In I.~Chatzigiannakis, C.~Kaklamanis, D.~Marx, and D.~Sannella, editors, {\em 45th International Colloquium on Automata, Languages, and Programming (ICALP 2018)}, volume 107 of {\em Leibniz International Proceedings in Informatics (LIPIcs)}, pages 115:1--115:14, Dagstuhl, Germany, 2018. Schloss Dagstuhl -- Leibniz-Zentrum f{\"u}r Informatik.

\bibitem{Ambos-SpiesK88}
K.~Ambos{-}Spies and J.~K{\"{a}}mper.
\newblock On disjunctive self-reducibility.
\newblock In E.~B{\"{o}}rger, H.~K. B{\"{u}}ning, and M.~M. Richter, editors, {\em {CSL} '88, 2nd Workshop on Computer Science Logic, Duisburg, Germany, October 3-7, 1988, Proceedings}, volume 385 of {\em Lecture Notes in Computer Science}, pages 1--13. Springer, 1988.

\bibitem{babaiGI}
L.~Babai.
\newblock Graph isomorphism in quasipolynomial time.
\newblock {\em arXiv:1512.03547}, 2015.

\bibitem{ballEJC31}
R.~N. Ball, J.~Ne\v{s}et\v{r}il, and A.~Pultr.
\newblock Dualities in full homomorphisms.
\newblock {\em European Journal of Combinatorics}, 31:106--199, 2010.

\bibitem{bangjensenJGT87}
J.~Bang-Jensen, J.~Huang, and X.~Zhu.
\newblock Completing orientations of partially oriented graphs.
\newblock {\em Journal of Graph Theory}, 87(3):285--304, 2017.

\bibitem{banachMFCS24}
M.~Bannach, F.~Chudigiewitsch, and T.~Tantau.
\newblock {On the Descriptive Complexity of Vertex Deletion Problems}.
\newblock In R.~Kr\'{a}lovi\v{c} and A.~Ku\v{c}era, editors, {\em 49th International Symposium on Mathematical Foundations of Computer Science (MFCS 2024)}, volume 306 of {\em Leibniz International Proceedings in Informatics (LIPIcs)}, pages 17:1--17:14, Dagstuhl, Germany, 2024. Schloss Dagstuhl -- Leibniz-Zentrum f{\"u}r Informatik.

\bibitem{BarsukovM23}
A.~Barsukov and F.~R. Madelaine.
\newblock On guarded extensions of {MMSNP}.
\newblock In G.~D. Vedova, B.~Dundua, S.~Lempp, and F.~Manea, editors, {\em Unity of Logic and Computation - 19th Conference on Computability in Europe, CiE 2023, Batumi, Georgia, July 24-28, 2023, Proceedings}, volume 13967 of {\em Lecture Notes in Computer Science}, pages 202--213. Springer, 2023.

\bibitem{bartoJACM68}
L.~Barto, J.~Bul\'{\i}n, A.~Krokhin, and J.~Opr\v{s}al.
\newblock Algebraic approach to promise constraint satisfaction.
\newblock {\em J. ACM}, 68(4), July 2021.

\bibitem{BellareGoldwasser}
M.~Bellare and S.~Goldwasser.
\newblock The complexity of decision versus search.
\newblock {\em {SIAM} J. Comput.}, 23(1):97--119, 1994.

\bibitem{bienvenu2014}
M.~Bienvenu, B.~ten Cate, C.~Lutz, and F.~Wolter.
\newblock {Ontology-Based Data Access: A Study through Disjunctive Datalog, CSP, and MMSNP}.
\newblock {\em {ACM} Trans. Database Syst.}, 39(4):33:1--33:44, 2014.

\bibitem{Book}
M.~Bodirsky.
\newblock {\em Complexity of Infinite-Domain Constraint Satisfaction}.
\newblock Lecture Notes in Logic (52). Cambridge University Press, Cambridge, United Kingdom; New York, NY, 2021.

\bibitem{BodirskyGrohe}
M.~Bodirsky and M.~Grohe.
\newblock Non-dichotomies in constraint satisfaction complexity.
\newblock pages 184 --196. Springer Verlag, July 2008.

\bibitem{ManuelSantiagoCSL}
M.~Bodirsky and S.~Guzm\'an-Pro.
\newblock Hereditary first-order logic: the tractable quantifier prefixes, 2025.
\newblock Preprint arXiv:2411.10860v2.

\bibitem{bodirskyJGT109}
M.~Bodirsky and S.~Guzmán-Pro.
\newblock The generic circular triangle-free graph.
\newblock {\em Journal of Graph Theory}, 109(4):426--445, 2025.

\bibitem{bodirsky_asnp}
M.~Bodirsky, S.~Kn{\"{a}}uer, and F.~Starke.
\newblock {ASNP:} {A} tame fragment of existential second-order logic.
\newblock In {\em CiE 2020, Proceedings}, volume 12098 of {\em Lecture Notes in Computer Science}, pages 149--162. Springer, 2020.

\bibitem{bondy2008}
J.~A. Bondy and U.~S.~R. Murty.
\newblock {\em Graph Theory}.
\newblock Springer, Berlin, 2008.

\bibitem{DecisionProblem}
E.~B{\"{o}}rger, E.~Gr{\"{a}}del, and Y.~Gurevich.
\newblock {\em The Classical Decision Problem}.
\newblock Perspectives in Mathematical Logic. Springer, 1997.

\bibitem{BG21}
J.~Brakensiek and V.~Guruswami.
\newblock Promise constraint satisfaction: Algebraic structure and a symmetric boolean dichotomy.
\newblock {\em {SIAM} J. Comput.}, 50(6):1663--1700, 2021.

\bibitem{BulatovFVConjecture}
A.~A. Bulatov.
\newblock A dichotomy theorem for nonuniform {CSP}s.
\newblock In {\em 58th {IEEE} Annual Symposium on Foundations of Computer Science, {FOCS} 2017, {B}erkeley, {CA}, {USA}, {O}ctober 15-17}, pages 319--330, 2017.

\bibitem{chusdnovskyAM164}
M.~Chudnovsky, N.~Roberston, P.~Seymour, and R.~Thomas.
\newblock The strong perfect graph theorem.
\newblock {\em Annals of Mathematics}, 164:51--229, 2006.

\bibitem{dantasDAM159}
S.~Dantas, C.~M. de~Figueiredo, M.~V. da~Silva, and R.~B. Teixeira.
\newblock On the forbidden induced subgraph sandwich problem.
\newblock {\em Discrete Applied Mathematics}, 159:1717--1725, 2015.

\bibitem{dantasENTCS346}
S.~Dantas, C.~M. de~Figueiredo, P.~Petito, and R.~B. Teixeira.
\newblock A general method for forbidden induced subgraph sandwich problem np-completeness.
\newblock {\em Electronic Notes in Theoretical Computer Science}, 346:393--400, 2019.

\bibitem{figueiredoDAM251}
C.~M. de~Figueiredo and S.~Spirkl.
\newblock Sandwich and probe problems for excluding paths.
\newblock {\em Discrete Applied Mathematics}, 251:146--154, 2018.

\bibitem{eiterSICOMP32}
T.~Eiter, G.~Gottlob, and K.~Makino.
\newblock New results on monotone dualization and generating hypergraph transversals.
\newblock {\em SIAM Journal on Computing}, 32(2):514--537, 2003.

\bibitem{eiterDAM156}
T.~Eiter, K.~Makino, and G.~Gottlob.
\newblock Computational aspects of monotone dualization: A brief survey.
\newblock {\em Discrete Applied Mathematics}, 156(11):2035--2049, 2008.
\newblock In Memory of Leonid Khachiyan (1952 - 2005 ).

\bibitem{evenIC61}
S.~Even, A.~L. Selman, and Y.~Yacobi.
\newblock The complexity of promise problems with applications to public-key cryptography.
\newblock {\em Information and Control}, 61(2):159--173, 1984.

\bibitem{Fagin}
R.~Fagin.
\newblock Generalized first-order spectra and polynomial-time recognizable sets.
\newblock {\em Complexity of Computation}, 7:43--73, 1974.

\bibitem{FederVardi}
T.~Feder and M.~Y. Vardi.
\newblock The computational structure of monotone monadic {SNP} and constraint satisfaction: {a} study through {D}atalog and group theory.
\newblock {\em {SIAM} Journal on Computing}, 28(1):57--104, 1999.

\bibitem{fominTCS64}
F.~Fomin, P.~Golovach, and D.~Thilikos.
\newblock On the parameterized complexity of graph modification to first-order logic properties.
\newblock {\em Theory of Computing Systems}, 64:251--271, 2020.

\bibitem{fredmanJA21}
M.~L. Fredman and L.~Khachiyan.
\newblock On the complexity of dualization of monotone disjunctive normal forms.
\newblock {\em Journal of Algorithms}, 21(3):618--628, 1996.

\bibitem{golumbicJA19}
M.~Golumbic, H.~Kaplan, and R.~Shamir.
\newblock Graph {S}andwich {P}roblems.
\newblock {\em Journal of Algorithms}, 19:449--473, 1995.

\bibitem{groheComACM63}
M.~Grohe and P.~Schweitzer.
\newblock The graph isomorphism problem.
\newblock {\em Commun. ACM}, 63(11):128–134, Oct. 2020.

\bibitem{hemaspaandraJCSS5}
E.~Hemaspaandra, A.~V. Naik, M.~Ogihara, and A.~L. Selman.
\newblock P-selective sets and reducing search to decision vs self-reducibility.
\newblock {\em Journal of Computer and System Sciences}, 5:194--209, 1996.

\bibitem{Hemaspaandra2020}
L.~A. Hemaspaandra.
\newblock {\em The Power of Self-Reducibility: Selectivity, Information, and Approximation}, pages 19--47.
\newblock Springer International Publishing, Cham, 2020.

\bibitem{Hodges}
W.~Hodges.
\newblock {\em A shorter model theory}.
\newblock Cambridge University Press, Cambridge, 1997.

\bibitem{Immerman}
N.~Immerman.
\newblock {\em Descriptive Complexity}.
\newblock Graduate Texts in Computer Science, Springer, New York, 1998.

\bibitem{koJCSS26}
K.-I. Ko.
\newblock On self-reducibility and weak p-selectivity.
\newblock {\em Journal of Computer and System Sciences}, 26:209--221, 1983.

\bibitem{KunNesetril24}
G.~Kun and J.~Nešetřil.
\newblock Dichotomy for orderings?, 2025.

\bibitem{LachlanWoodrow}
A.~H. Lachlan and R.~E. Woodrow.
\newblock Countable ultrahomogeneous undirected graphs.
\newblock {\em Transactions of the AMS}, 262(1):51--94, 1980.

\bibitem{ladnerACM22}
R.~E. Ladner.
\newblock On the structure of polynomial time reducibility.
\newblock {\em J. ACM}, 22(1):155–171, Jan. 1975.

\bibitem{macphersonDM311}
D.~Macpherson.
\newblock A survey of homogeneous structures.
\newblock {\em Discrete Mathematics}, 311(15):1599--1634, 2011.
\newblock Infinite Graphs: Introductions, Connections, Surveys.

\bibitem{maghout1959}
K.~Maghout.
\newblock Sur la determination des nombres de stabilite et du nombre chromatique dun graphe.
\newblock {\em COMPTES RENDUS HEBDOMADAIRES DES SEANCES DE L ACADEMIE DES SCIENCES}, 248(25):3522--3523, 1959.

\bibitem{meyerMIT79}
A.~Meyer and M.~Patersons.
\newblock {With what frequency are apparently intractable problems difficult?}, 1979.
\newblock {Tech. Report MIT/LCS/TM-126, Massachusetts Institute of Technology, Cambridge, MA}.

\bibitem{ICALP-1976-Schnorr}
C.-P. Schnorr.
\newblock {Optimal Algorithms for Self-Reducible Problems}.
\newblock In {\em {Proceedings of the Third International Colloquium on Automata, Languages and Programming}}, pages 322--337, 1976.

\bibitem{selmanIC78}
A.~L. Selman.
\newblock Promise problems for complexity classes.
\newblock {\em Information and Computation}, 78:87--98, 1988.

\bibitem{straubingSIGLOG5}
H.~Straubing.
\newblock First-order logic and aperiodic languages: a revisionist history.
\newblock {\em ACM SIGLOG News}, 5(3):4–20, July 2018.

\bibitem{ZhukFVConjecture}
D.~N. Zhuk.
\newblock A proof of {CSP} dichotomy conjecture.
\newblock In {\em 58th {IEEE} Annual Symposium on Foundations of Computer Science, {FOCS} 2017, {B}erkeley, {CA}, {USA}, {O}ctober 15-17}, pages 331--342, 2017.
\newblock https://arxiv.org/abs/1704.01914.

\end{thebibliography}

\newpage
\appendix

\section{A remark concerning the definition of extensional ESO}

We point out that if we remove the condition ``and $R'$ does not appear
anywhere else in $\Psi$'' in the definition of extensional $\ESO$ (in Section~\ref{sec:extensional-ESO}),
we obtain essentially the same logic as ESO.

\begin{observation}\label{obs:meaningless}
    For every $\ESO$ $\tau$-sentence $\Phi$, there are a signature $\tau^\ast\supseteq \tau$ and
    an $\ESO$ $\tau^\ast$-sentence $\Psi$ such that
    \begin{itemize}
        \item a $\tau^\ast$-structure $\bA$ satisfies $\Psi$ if and only if $\bA$ is a $\tau$-structure
        and $\bA\models \Phi$,
        \item if $\Phi$ is monotone (resp.\ in $\SNP$), then $\Psi$ is monotone (resp.\ in $\SNP$), and
        \item for every existentially quantified symbol $R$ of $\Psi$, there
        is a distinct symbol  $R'\in \tau^\ast$ such that $\Psi$ contains the clause
        $\forall \overline{x}.R'(\overline{x})\Rightarrow R(\overline{x})$.
    \end{itemize}
\end{observation}
\begin{proof}
    Let $\phi$ be the first-order part of $\Phi$; we assume $\phi$ is in conjunctive
    normal form. 
    We first consider the case where $\Phi$ holds on every finite $\tau$-structure. In this case,
    let $\psi$ be the first-order $\tau$-sentence obtained from $\phi$ by removing all clauses that
    contain an existentially quantified symbol from $\Phi$. Since $\Phi$ holds on all finite
    $\tau$-structures, $\psi$ does as well. Hence, $\Phi$ and $\psi$ are equivalent on finite
    structures. Moreover, if $\Phi$ is monotone (resp.\ universal) then so is $\psi$, and $\psi$ satisfies
    the third itemized statement because $\psi$ has no quantified symbols. 

    Let $\sigma$ be the set of existentially quantified symbols in $\Phi$. Consider the
    relational signature $\tau^\ast$ that contains $\tau$ and additionally a symbol $R'$ of arity $r$
    for every $R\in \sigma$ of arity $r$. Let $m$ be the maximum
    arity of the quantified symbols of $\Phi$, and define $\psi$ as
    \[ 
    \psi:= \phi \land \forall x_1,\dots, x_m \left ( \bigwedge_{R \in \sigma} \lnot R'(x_1,\dots, x_r)  \land \left( \bigwedge_{R\in \sigma} R'(x_1,\dots, x_r) \Rightarrow R(x_1,\dots, x_r) \right ) \right ).
    \]
    Note that this formula can be simplified; the purpose of writing it in the given redundant form is to
    make sure that the formula $\Psi$ that we define next meets the syntactic requirements from the observation.
    We now define $\Psi$ as the $\ESO$ $\tau^\ast$-formula obtained from $\Phi$ be replacing
    $\phi$ by $\psi$. 
    Since $\tau\subseteq \tau^\ast$, every $\tau$-structure $\bA$ is 
    a $\tau^\ast$-structure. It is also straightforward to observe that $\bA$ satisfies
    $\Phi$ if and only if it satisfies $\Psi$. Hence, 
    the forward implication
    of the first itemized
    statement holds for $\Phi$ and $\Psi$. To see that the converse also holds consider a
    $\tau^\ast$-structure $\bB$. Notice that if the interpretation of $R'$ in $\bB$
    is not empty for some $R\in \sigma$, then $\bB\not\models \Psi$.
    Otherwise, $\bB$ is a $\tau$-structure, and it follows from the definition
    of $\Psi$ that $\bB\models \Phi$ if and only if $\bB\models \Psi$. We thus
    conclude that the first itemized statement holds for $\Phi$ and $\Psi$. The second and third
    itemized statement hold by the definition of $\Psi$. 
\end{proof}

\section{Monotone connected extensional SNP}
\label{ap:fdCSP->mceSNP}

The following lemma is used in our proof of Theorem~\ref{thm:muSNP-fbCSP-sucSNPwe}.

\begin{lemma}\label{lem:B->BB2}
    For every extensional $\SNP$ $\tau$-sentence $\Phi$ and a binary symbol $E'$ not in $\tau$
    there is a connected extensional $\SNP$ $(\tau\cup\{E'\})$-sentence $\Psi$ such that the following statements
    hold.
    \begin{itemize}
        \item There is a polynomial-time reduction from $\Phi$ to $\Psi$, and a polynomial-time conjunctive reduction
        from $\Psi$ to $\Phi$. 
        \item If $\Phi$ is preserved under disjoint unions, then $\Phi$ and $\Psi$ are log-space equivalent.
        \item If $\Phi$ is in monotone $\SNP$, then $\Psi$ is in monotone $\SNP$.
    \end{itemize}
\end{lemma}
\begin{proof}
    We follow almost the same proof as the one for Proposition 1.4.11 in~\cite{Book} ---
    we thus choose to refer the reader to that proof for any missing details in this one. Let $\tau$ be a 
    finite relational signature and let $\Phi$ be an extensional $\SNP$ $\tau$-sentence of the 
    form $$\exists R_1,\dots, R_k \forall x_1,\dots, x_n  (\nu \land \phi)$$ where $\nu$ consists of the
    conjuncts $R_i'\Rightarrow R_i$ for every $i\in [k]$  and $\phi$ is a quantifier-free 
    $(\tau\cup\{R_1,\dots, R_k\})$-formula in conjunctive normal form.
    We consider two new binary relation symbols $E'$ and $E$ and the signature $\tau' := \tau \cup\{E'\}$.
    We define the uniform  $\SNP$ $\tau'$-sentence
    $$\Psi:= \exists R_1,\dots, R_k, E \; \forall x_1,\dots, x_n, z_1,z_2,z_3 (\nu\land E'(z_1,z_2)
    \Rightarrow E(z_1,z_2) \land \psi)$$ where $\psi$ is the conjunction of the following clauses: 
    \begin{itemize} 
        \item  $\lnot E(z_1,z_2) \lor \lnot E(z_2,z_3)\lor E(z_1,z_3)$ (i.e., $E$ is a transitive relation),
        \item $\lnot E(z_1,z_2) \lor E(z_2,z_1)$ (i.e., $E$ is a symmetric relation),
        \item for every $R\in \tau$ of arity $r$ and all $i, j \in  [r]$ where $i\neq j$, the formula $\psi$ contains the conjunct
        $R(x_1,\dots, x_r) \Rightarrow E(x_i,x_j)$, and
        \item for each clause $\phi'$ of $\phi$ with free variables $y_1,\dots, y_m \subseteq \{x_1,\dots, x_n\}$, the formula 
        $\psi$ contains the conjunct
        \[
        \phi' \lor \bigvee_{i,j\in [m]} \lnot E(y_i,y_j).
        \]
    \end{itemize}
    Clearly, $\Psi$ is connected. Since $\Phi$ is in extensional SNP, $\Psi$
    is in extensional SNP as well. It is not hard to verify that a $\tau'$-structure $(\bB,E')$ satisfies
    $\Psi$ if and only if the $\tau$-reduct $\bC$ of every connected component $(\bC,E')$ of $(\bB,E')$ satisfies
    $\Phi$. Thus, there is a polynomial-time conjunctive reduction from $\Psi$ to $\Phi$. 
    If $\Phi$ is preserved under disjoint unions, then $(\bB,E')\mapsto \bB$ is a polynomial-time
    (and log-space) reduction from  $\Psi$ to $\Phi$. For the converse reduction, on input structure $\bA$
    consider the $(\tau\cup\{E'\})$-structure $\bB:=(\bA, A^2)$ and notice that $\bA\models \Phi$ if and only
    if $\bB \models \Psi$. Hence, $\Phi$ reduces in logarithmic space to $\Psi$.
    Finally, it follows from the definition of $\Psi$ that if $\Phi$ is monotone, then $\Psi$ 
    is monotone as well.
\end{proof}

\subsection*{Finite-domain CSPs}
This appendix contains a lemma which in particular implies that every finite domain CSP is
expressible in connected monotone extensional $\SNP$. Given a relational $\tau$-structure
$\bA$, we denote by $\Age(\bA)$ the class of finite structures that embed into $\bA$;
equivalently, the class of finite substructures of $\bA$, up to isomorphism. A set of
finite structures $\calF$ is called a set of \emph{bounds} of $\bA$, 
if for every finite structure $\bB$ it is the case that $\bB\in \Age(\bA)$ if and only if $\bB$
is $\calF$-free. We say that $\calF$ is a set of \emph{connected} bounds if every structure in
$\calF$  is connected. The \emph{Gaifman graph} of $\bA$ is the undirected 
graph $G(\bA)$ with vertex set $A$ and there is an edge $ab$ if $a$ and $b$ belong to some common
tuple in $R^\bA$ for some $R\in \tau$. An \emph{induced path} in $\bA$ is a sequence of vertices
$a_1,\dots, a_n$ that induce a path in the Gaifman graph of $\bA$. In this case, we say that
the path is an \emph{$a_1a_n$-path}, and that the \emph{length} of this path is $n$. It is
straightforward to observe that for every finite relational signature $\tau$ and every positive
integer $N$, there is a finite set of $\tau$-structures $\calP_N$ such that a $\tau$-structure
$\bA$ has no induced path of length $N$ if and only if $\bA$ is $\calP_N$-free.

\begin{lemma}\label{lem:bounded-path1}
    Let $\bA$ be a finitely bounded structure. If there is a positive integer $N$ such that
    every induced path in $\bA$ has length strictly less than $N$, then there is a structure 
    $\bB$ such that
    \begin{itemize}
        \item $\Age(\bA) \subseteq \Age(\bB)$, 
        \item a structure $\bC$ belongs to $\Age(\bB)$ if and only if every connected component of
        $\bC$ belongs to $\Age(\bA)$, and
        \item there is a finite set of connected bounds for $\Age(\bB)$.
    \end{itemize}
\end{lemma}
\begin{proof}
    Let $\calF$ be a finite set of bounds of $\bA$, and let $\calP_N$ be as in the paragraph above.
    Let $\calF_C$ be the set of connected structures in $\calF$, and for every disconnected
    structure $\bF\in \calF$ where $\bF := \biguplus_{i = 1}^n \bF_n$
    and each $\bF_n$ is a connected structure, let $\calC(\bF)$ be the class of finite $\tau$-structures
    $\bC$  such that
    \begin{itemize}
        \item $\bF$ embeds into $\bC$, and
        \item for all distinct $i,j \in \{1,\dots,n\}$, $a\in F_i$, and $b\in F_j$, there is an induced $ab$-path
        in $\bC$ of length strictly less than $N$. 
    \end{itemize}
    Notice that $\calC(\bF)$ has, up to isomorphism, a finite set of minimal structures $\mathbb M$ 
    with respect to the embedding order. Indeed, let $\mathbb M$ be such a minimal structure, and
    identify $\bF$ with a substructure of $\mathbb M$ (isomorphic to 
    $\bF$). For every $a,b\in F$ that belong to different components of $\bF$, let $P_{ab}$ be the set of vertices
    witnessing that there is an induced $ab$-path of length less than $N$ in $\mathbb M$. By the minimality of $\mathbb M$, it must be
    the case that every element of $\mathbb M$  belongs to $F$ or to $P_{ab}$ for some $a,b\in F$ (that belong to different
    components of $\bF$). Hence, the cardinality of $M$ is bounded by $|F|$, plus $|F|^2$ times the maximum number $m$ of vertices
    in an induced path of length $N-1$ --- notice that the last $m$ only depends on the signature of $\mathbb M$ and
    on $N$, so it does not depend on $\mathbb M$. 
    
    We denote this finite set of minimal structures in $\calC(\bF)$ by $\min(\calC(\bF))$. 
    Finally, let
    \[
    \calF^\ast:= \calF_C\cup \calP_N \cup \bigcup_{\bF\in \calF\setminus \calF_C}\min(\calC(\bF)).
    \]
    Since $\calF^\ast$ consists of connected structures, the class of $\calF^\ast$-free structures
    is closed under disjoint unions, and so, there is a structure $\bB$ whose age is the class
    of finite $\calF^\ast$-free structures (see, e.g.,~\cite[Proposition 2.3.1]{Book}). 
    We claim that $\bB$ satisfies the itemized statements of the lemma.
    By construction, it satisfies the third one. To show the first one, 
    let $\bC$ be a finite $\tau$-structure that embeds into $\bA$. In particular,
    $\bC$ is $\calF$-free. Moreover, $\bA$ is $\calP_N$-free because
    $\bA$ has no induced path of length $N$, and thus $\bC$ is also $\calP_N$-free. 
    Hence, $\bC$ is  $\calF^\ast$-free, and so $\bC\in \Age(\bB)$. 
    
    To show the second itemized statement, 
    first observe that if every connected component of a finite structure $\bC$ belongs to $\Age(\bA)$, then $\bC$ is a disjoint
    unions of structures that belong to $\Age(\bA) \subseteq \Age(\bB)$. Since $\bB$ has a set of connected bounds, its age is closed
    under disjoint unions, and so $\bC\in \Age(\bB)$. 
    For the converse implication, it suffices to prove that every connected structure $\bC$ in the age of $\bB$ belongs to the age of $\bA$.
    Since $\bC$ is an $\calF^\ast$-free structure, it contains no induced path of length
    larger than $N$. Now, suppose for contradiction that $\bC\not\in \Age(\bA)$, so there is an embedding of some $\bF\in \calF$ into $\bC$.
    Moreover, since $\calF_C\subseteq \calF^\ast$, it must be the case that $\bF$
    is not connected. But since $\bC$ is connected and has no path of length larger than $N$, we must have
    $\bC \in \calC(\bF)$. This implies that $\bC$  contains some structure from $\min(\calC(\bF))\subseteq \calF^\ast$, 
    contradicting the assumption that $\bC$ is $\calF^\ast$-free. The claim now follows. 
\end{proof}

\begin{lemma}\label{lem:boundedpath2}
    Let $\bA$ be a finitely bounded structure. If there is a positive integer $N$ such that
    every induced path in $\bA$ has length strictly less than $N$, then $\CSP(\bA)$ is in monotone
    connected extensional $\SNP$.
\end{lemma}
\begin{proof}
        Let $\bB$ be the structure from Lemma~\ref{lem:bounded-path1} for $\bA$. It follows from the 
        second itemized property of $\bB$ that
        $\CSP(\bA) = \CSP(\bB)$. Since $\bB$ has a finite set of connected bounds by the third itemized property of $\bB$, the claim of the lemma
        follows from Corollary~\ref{cor:FBstructures}.
\end{proof}

\begin{corollary}\label{cor:finite-domain-MCESNP}
   If $\bA$ is a finite structure, then $\CSP(\bA)$ is in monotone connected extensional $\SNP$. 
\end{corollary}

\end{document}